\documentclass[leqno]{amsart}
\usepackage{amscd,amssymb,amsmath}
\usepackage[english,vietnamese]{babel}    

\usepackage[dvipsnames]{xcolor}
\colorlet{deJP}{blue!40!red}

\usepackage{tikz}
\usepackage{float}
\usepackage{graphicx}
\usepackage{url} 
\usetikzlibrary{patterns}

\usepackage{tikz}
\usepackage{float}
\usepackage{graphicx}
\usepackage{url} 
\usetikzlibrary{patterns}

\usepackage{amsmath,amsfonts,amsthm,amssymb}
\usepackage[all]{xy}     
\numberwithin{equation}{section}

\theoremstyle{definition}
\newtheorem{theorem}[equation]{Theorem}

\newtheorem{definition}[equation]{Definition}
\newtheorem{remark}[equation]{Remark}

\theoremstyle{definition}

\newtheorem{lem}[equation]{Lemma}
\newcommand{\Z}{\mathbb{Z}}
\newcommand{\R}{\mathbb{R}}
\newcommand{\Sp}{\mathbb{S}}
\newcommand{\T}{\mathbb{T}}
\newcommand{\PP}{\mathbb{P}}

\renewcommand{\baselinestretch}{1.2}
\definecolor{lightgray}{gray}{0.65}

\begin{document}

\title[An elementary proof of Euler's formula using Cauchy's Method]{An elementary 
proof of Euler's formula \\ using Cauchy's Method}

\author{Jean-Paul \textsc{Brasselet} and Nguy\~\ecircumflex n Th\d{i} B\'ich Th\h{u}y
}
\address{CNRS I2M and Aix-Marseille Universit\'e, Marseille, France.}
\email{jean-paul.brasselet@univ-amu.fr}
\address{UNESP, Universidade Estadual Paulista, ``J\'ulio de Mesquita Filho'', S\~ao Jos\'e do Rio Preto, Brasil}
\email{bich.thuy@unesp.br}

\newenvironment{preuve}{{\noindent \sc {\bf Proof.}} } {\mbox{ }\hfill $\Box$}
\renewcommand{\abstractname}{Abstract}
\renewcommand{\figurename}{Figure}

\begin{abstract}

The use of Cauchy's method to prove Euler's well-known formula is an object of many controversies. 
The purpose of this  paper is to prove that Cauchy's method applies for convex polyhedra 
and not only for them,  but also for 
surfaces such as the torus, the projective plane, the Klein bottle and the pinched torus.
\end{abstract}

\maketitle
 
\section*{Introduction}

 Euler's formula says that for any convex polyhedron the alternating sum 
\begin{equation}\label{carEuler}
n_0 - n_1 + n_2, 
\tag{1}
\end{equation}
is equal to 2, where the numbers $ n_i $ are respectively the number of  vertices $ n_0 $,
the number of edges $n_1 $ and the number of triangles  $n_2 $ of the polyhedron.
There are many controversies about the paternity of the formula, 
also about who gave the first correct proof.

In section \ref{histoire}, we provide information about the history of the formula as well as 
about the first topological proof given by Cauchy. Some authors criticize Cauchy's proof,
 saying that the proof needs deep topological results that were proved after Cauchy's time: 
 ``{\it N\~ao se pode, portanto, esperar obter uma demonstra\c c\~ao elementar do Teorema 
 de Euler, com a hip\'otese de que o poliedro \'e homeomorfo a uma esfera,
  como fazem Hilbert-Cohn Vossen e Courant-Robbins}'' (Lima, \cite{Li2}) 
  \footnote{``Therefore, we cannot expect to obtain an elementary proof of Euler's theorem, 
  with the hypothesis that the polyhedron is homeomorphic to a sphere, as Hilbert-Cohn Vossen 
  and Courant-Robbins did.''}. 
 Notice that the proof provided by Hilbert-Cohn Vossen \cite{HC} 
 and  Courant-Robbins \cite{CR}  is the one of Cauchy. 
 
In section \ref{OTeo}, we 
provide an elementary proof which shows that only with a lifting technique and the use of 
sub-triangulations, Cauchy's proof works without using any other result. More precisely, 
considering a triangulated polygon in the plane, with possible identifications of the simplices on its boundary, 
 we prove that  the alternating sum (\ref{carEuler})  
of the polygon is equal to the one of its boundary plus $1$ (Theorem \ref{teo1}). The idea 
of our  proof is, starting from the hole formed by the removal of a simplex, to extend the hole 
by successive puddles.
The process is illutrated by the construction of a suitable pyramid.
A direct consequence of the theorem is  an elementary proof of Euler's formula 
using only Cauchy's method.

 As applications of our theorem \ref{teo1}, in section  \ref{applications}, 
we also use these tools to prove that for a triangulable surface ${\mathcal S}$ 
like the torus, the projective plane, the Klein bottle and even for the pinched torus, the alternating sum
(\ref{carEuler}) does not depend on the triangulation of the surface. 
 To be completely honest, for applications (other than the sphere) in section \ref{applications},  
we also use the idea of
``cutting'' surfaces that, in general, was introduced by Alexander Veblen in a seminar in 1915 
(see \cite{Bra}). 
Of course, one can ask why we do not apply theorem \ref{teo1} and the 
same reasoning to all (smooth) orientable and non-orientable surfaces. The reason is 
very simple: we want to provide proofs that it was possible to do at the time of Cauchy. 
It is only in 1925 that T. Rad\'o \cite{Rad} proved the triangulation theorem 
for surfaces, that was more or less assumed in Cauchy's time. The classification theorem
for compact surfaces and the representation by the ``normal'' form was proved 
for the first time in a rigorous way by H.R. Brahana \cite{Bra} (1921). 
It is clear that using our theorem \ref{teo1}  and the representation of surfaces under the normal form, 
we immediately obtain the Euler-Poincar\'e characteristic of any compact surface. 
However, that is like a dog biting his own tail. That it is the reason we do not 
present the result for surfaces in general but only what is possible to do with Cauchy's method
 for some elementary surfaces.

The first author had financial support of FAPESP 
(process UNESP-FAPESP  number 2015/06697-9).

\bigskip

\section{History} \label{histoire}

\subsection{Before Cauchy} \label{before}

The name ``Euler's formula'' comes from an announcement of Leonhard Euler in    November 14th of 1750
in a letter to his friend Goldbach of the following result:

\begin{theorem} \label{Euler's Theorem}
Let $K$ be a convex polyhedron, with $n_0$ vertices, $n_1$  edges  and $n_2$  two-dimensional polygons, then 
\begin{equation} \label{Euler}
n_0 - n_1 + n_2 = 2.
\tag{2}
\end{equation}
\end{theorem}

There are many different possible definitions of polyhedra.  
The discussion concerns the dimension of a polyhedron: Is a polyhedron a solid object 
of dimension three  or only its surface? 
In this paper, we call ``polyhedron'' the three dimensional solid figure.
A polyhedron is a figure constructed by polygons 
in such a way that each segment is the common face of exactly two 2-dimensional polygons 
and each vertex is the common face of at least three segments  (see \cite{Ri}, Chapter 2 for discussion).

\smallskip

 There are many questions and controversies about Euler's formula.
Here, let us discuss the two following questions: 
Was Euler the first mathematician stating the formula? 
Who provided the first proof of the formula?

\medskip 

Let us discuss the first question: {\it Who was the first to state the formula?}

\medskip 

Some authors 
(see \cite{Eve},  \S 3.12; \cite{Lie}, p. 90) write that it is possible that  
Archimedes ($\sim$ 287 AC, $\sim$ 212 AC)  already knew the formula. 
Some authors say that the formula was 
known to Descartes (1596 -- 1650). 
In fact, Descartes, in a manuscript \cite{De}, {\it ``De solidorum elementis''}, 
proved the following result:

\begin{theorem}\label{Descartes}
The sum of the angles of all the 2-dimensional polygons 
of a convex polyhedron is equal to $2(n_0 -2) \pi$.
\end{theorem}

\begin{proof} [{\bf Proof that ``formula (\ref{Euler}) is equivalent to theorem \ref{Descartes}''}] 
Let us denote by $i=1, \ldots, n_2$ the 2-dimen\-sional faces of a convex polyhedron. 
For each face $i$, let us denote by $k_i$ the number of vertices, which is also the number 
of edges of the face. 
We use, for each face, the following property: In a convex polygon with $k_i$ edges,  
the sum of all the angles equals $(k_i -2)\pi$. 
Since each edge of the polyhedron appears in two faces of the polyhedron, 
then $\sum_{i=1}^{n_2} k_i = 2n_1$. 
Hence the sum of the angles of all the faces of the polyhedron equals 
$\sum_{i=1}^{n_2} (k_i -2)\pi$, that is   $(2n_1 - 2n_2) \pi$. 
We obtain equivalence between theorem \ref{Descartes} and formula (\ref{Euler}). 
\end{proof}

Descartes did not publish his manuscript. 
The original version of the manuscript disappeared, but a copy was rediscovered in 1860 
among papers left by Leibnitz  (1646-1716). 
This copy suffered some accidents, 
in particular an immersion in the Seine river in Paris 
(see  \cite{Fo}, \cite{dJ1}).  

Some authors say that Descartes discovered the formula  (\ref{Euler})  as an application 
of his theorem \ref{Descartes}. 
This is a reason why sometimes the formula   (\ref{Euler}) is called the  
``{\it Descartes-Euler formula}''. 
Other authors, for example, Malkevich \cite{Mal1}, 
affirmed that  ``{\it 
Though Descartes did discover facts about 3-dimensional polyhedra that would have enabled him to deduce Euler's formula, he did not take this extra step. With hindsight it is often difficult to see how a talented mathematician of an earlier era did not make a step forward that with today's insights seems natural, however, it often happens.}'' 
 However, we know that some of Descartes' papers disappeared, 
 so nobody can decide if Descartes knew the formula or not and  the response 
 to the first question is not known.

\medskip 

Let us now discuss the second question: 
 {\it Who was the first to provide a correct proof of formula  (\ref{Euler})?}

\medskip 

Two years after writing the formula, 
Euler provided a proof  \cite{Eu2,Eu3} 
which consisted of removing step by step each vertex of the polyhedron together with the pyramid of which it is the vertex. 

Here we provide the example  of the cube, copied 
from \cite{Ri} : In figure \ref{dEuler} (b), 
one eliminates the vertex $A$ as well as the (white) pyramid of which $A$ is a vertex. 
This operation does not change the alternating sum $n_0-n_1+n_2$. 
In figure \ref{dEuler} (c) 
we perform the same process in order to eliminate the vertex $B$, and so on, till
we obtain a tetrahedron. For the tetrahedron, one has $n_0-n_1+n_2 = 2$, so we obtain the formula. 

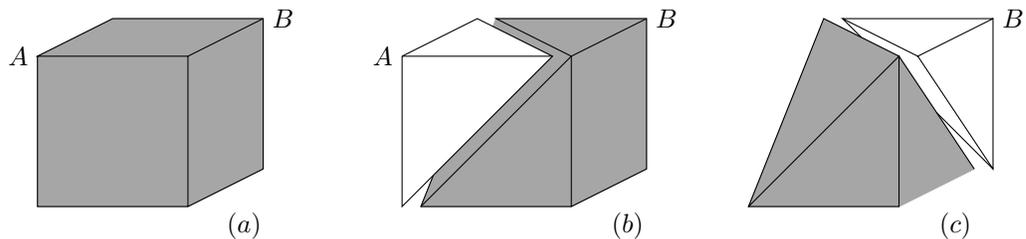
\begin{figure} [H] 
\begin{tikzpicture}  [scale= 0.5] 

\coordinate (A) at (0,0);
\coordinate (B) at (4,0);
\coordinate  (C) at (6,1);
\coordinate  (D) at (0,4);
\coordinate (E) at (4,4);
\coordinate (F) at (2,5);
\coordinate (G) at (6,5);

\filldraw [lightgray] (A) -- (D) -- (F) -- (G) -- (C) -- (B) -- (A);
\draw (D)--(A) -- (B) -- (C)--(G)--(F) -- (D) -- (E) --(B);
\draw (E) -- (G) ;
\node at (0,4)[left]{$A$};
\node at (6,5)[right]{$B$};
\node at (5.5,-0.5){$(a)$};
\end{tikzpicture}  
\qquad 
 \begin{tikzpicture}  [scale= 0.5] 

\coordinate (A) at (0,0);
\coordinate (B) at (4,0);
\coordinate  (C) at (6,1);
\coordinate  (D) at (0,4);
\coordinate (E) at (4,4);
\coordinate (F) at (2,5);
\coordinate (G) at (6,5);
\coordinate  (H) at (-0.5,0);
\coordinate (J) at (-0.5,4);
\coordinate (K) at (1.5,5);
\coordinate (L) at (3.5,4);

\coordinate (M) at (intersection of A--F and H--L);
\coordinate (N) at (intersection of K--L and A--F);

\filldraw [lightgray] (A) -- (M) -- (L) -- (N) -- (F) -- (G) -- (C) -- (B) -- (A);
\draw (M)--(A) -- (B) -- (C)--(G)--(F) --  (E) --(B);
\draw (G) -- (E) -- (A) ;
\draw (J) -- (H) -- (L) -- (K) -- (J) ;
\draw (J) --   (L)  ;
\node at (-0.5,4)[left]{$A$};
\node at (6,5)[right]{$B$};
\node at (5.5,-0.5){$(b)$};
\end{tikzpicture}  
\qquad
\begin{tikzpicture}  [scale= 0.5] 

\coordinate (A) at (0,0);
\coordinate (B) at (4,0);
\coordinate  (C) at (6,1);
\coordinate  (D) at (0,4);
\coordinate (E) at (4,4);
\coordinate (F) at (2,5);
\coordinate (G) at (6,5);
\coordinate (P) at (2.5,5);
\coordinate (Q) at (4.5,4);
\coordinate  (R) at (6.5,5);
\coordinate  (S) at (6.5,1);
\coordinate (T) at (intersection of F--E and P--S);
\coordinate (U) at (intersection of C--E and P--S);

\filldraw [lightgray] (A) -- (F) -- (E) -- (C) -- (B) -- (A);

\draw (A)--(F) --  (E) --(A) -- (B)-- (E)-- (C) ;
\draw (R)--(P) -- (Q) -- (R) -- (S)--(Q);
\draw(P)--(T);
\draw (U)--(S);
\node at (6.5,5)[right]{$B$};
\node at (5.5,-0.5){$(c)$};
\end{tikzpicture}  

\caption{Euler's proof: Successive elimination of a vertex as well as the pyramid 
 of  which  it is the vertex.} \label{dEuler}
\end{figure}

But this proof is not correct. 
In the book \cite{Ri}, Richeson provides  
 a clear description of Euler's proof, 
as well as the problems with the proof. According to Richeson, 
 these problems were solved by  
 Samelson  and by Francese and Richardson (\cite{Sam, FR}).  

The first correct proof was provided by Legendre \cite{Leg}  
in the first edition of his book 
{\it \'El\'ements de G\'eom\'etrie} (1794)  
(see  \cite{Ri}, Chapter 10 for a presentation of Legendre's proof). 
Legendre's argument was geometric in the same way as the 
proof of  Descartes for theorem \ref{Descartes}. 
The only difference between Descartes' argument and Legendre's argument
is that Descartes used the sphere presentation of the polyhedron (polar polyhedron), 
while Legendre used the polyhedron itself. 
The passage of the polyhedron  $K$ (Legendre) to the polar polyhedron of  $K$ (Descartes) 
makes a permutation of $n_0$ and $n_2$, without modifying $n_1$. 
 This is the reason why some authors say that the proof of 
 Euler's formula (\ref{Euler}) 
should be called ``Descartes-Legendre's proof''.

\subsection{Cauchy's time: Cauchy's method - The first combinatorial and topological proof} \label{methode_Cauchy}

In February of 1811, then 22 years old, Cauchy, who was already an engineer
of  {\it Ponts et Chauss\'ees}, gave a talk entitled  {\it Recherches sur les poly\`edres}  
at the {\it \'Ecole Polytechnique}, in Paris. 
This talk was published in 1813 in the {\it Journal de l'\'Ecole Polytechnique} \cite{Ca1}, as the first  combinatorial and topological proof of Euler's formula  (\ref{Euler}). 
This nice proof of Cauchy is included in many books (see, in particular, \cite{Ri}, Chapter 12
for a presentation with comments). 

The first step of Cauchy's proof is to construct a planar representation. 

\begin{definition} \label{repplanar}
 A planar  representation of a compact and without boundary surface  
$\mathcal{S}$ is a triple $(K, K_0, \varphi)$ where:
\begin{enumerate}
\item  $K$ is a 2-dimensional polygon in $\R^2$, 
\item the segments and vertices of the boundary of $K$ are named and oriented with possible identifications. 
We denote by $K_0$ the boundary of  $K$ with the given identifications,
\item $\varphi: \vert K \vert \to \mathcal{S}$ is a homeomorphism of the geometrical realisation of   $K$ (taking care of identifications of the segments and vertices on the boundary $K_0$) onto $\mathcal{S}$. 
\end{enumerate}
\end{definition}

It seems that  Cauchy was the first person to use the idea of planar representation. 
We now present Cauchy's proof. 

\medskip 

\begin{preuve}  \label{proof_Cauchy} 
Given a convex polyhedron $\widehat K$, 
we choose  a  2-dimensional face 
${\mathcal P}$ of the polyhedron.  We remove this face. 
The first idea of Cauchy is to construct the associated planar representation $K$ 
 of the polyhedron 
with respect to the choice of the removed face.
 Lakatos \cite{Lak} explains Cauchy's construction as the following way: 
 Put a camera above the removed face of the polyhedron, 
the planar representation will appear as the photograph. 
Notice that this idea of planar representation is similar to stereographic projection.

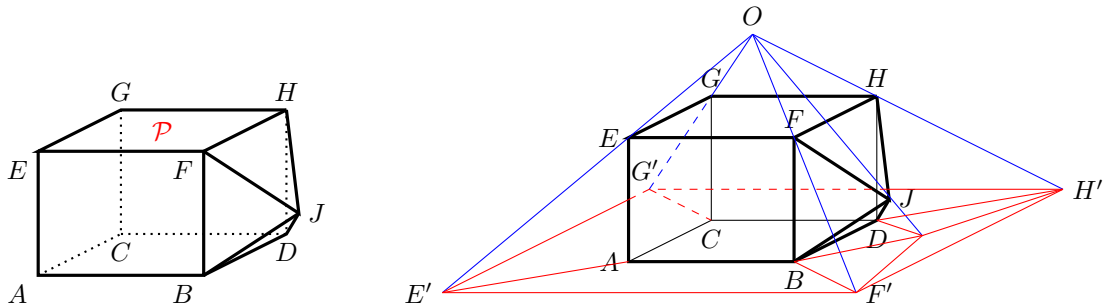
\begin{figure} [H]
\centering
\begin{tikzpicture}  [scale= 0.55] 

\coordinate (A) at (0,0);
\coordinate (B) at (4,0);
\coordinate  (C) at (2,1);
\coordinate  (D) at (6,1);
\coordinate (E) at (0,3);
\coordinate (F) at (4, 3);
\coordinate  (G) at (2,4);
\coordinate  (H) at (6,4);
\coordinate (J) at (6.3,1.5);

\node at (-0.5,0)[below]{$A$};
\node at (3.5,0)[below]{$B$};
\node at (C)[below]{$C$};
\node at (D)[below]{$D$};
\node at (-0.5,3)[below]{$E$};
\node at (3.5, 3)[below]{$F$};
\node at (G)[above]{$G$};
\node at (H)[above]{$H$};
\node at (6.3,1.5)[right]{$J$};

\node at (3,3.5) [red]{${\mathcal P}$};

\draw [very thick] (E) -- (A) -- (B) -- (F) -- (E) -- (G) -- (H) ; 
\draw [very thick] (H) -- (J) --  (F)--(H) ;
\draw [very thick] (B) -- (J);
\draw [very thick] (J) -- (D) -- (B);

\draw [thick,dotted] (A) -- (C) --(G);
\draw [thick,dotted] (C) -- (D) --(H);

\end{tikzpicture}   
\qquad
\begin{tikzpicture}  [scale= 0.55] 

\coordinate (A) at (0,0);
\coordinate (T) at (4,0);
\coordinate  (C) at (2,1);
\coordinate  (D) at (6,1);
\coordinate (E) at (0,3);
\coordinate (F) at (4, 3);
\coordinate  (G) at (2,4);
\coordinate  (H) at (6,4);
\coordinate (J) at (6.3,1.5);
\coordinate (J2) at (6.3,1.55);

\coordinate  (K) at (3,5.5);
\coordinate (L) at (-4.5,-0.75);
\coordinate (M) at (5.5,-0.75);
\coordinate  (N) at (0.5,1.75);
\coordinate  (O) at (10.5,1.75);
\coordinate (P) at (7.1,0.63);

\coordinate (R) at (intersection of N--L and A--E);
\coordinate (S) at (intersection of H--J and N--O);

\node at (A)[left]{$A$};
\node at (T)[below]{$B$};
\node at (C)[below]{$C$};
\node at (D)[below]{$D$};
\node at (E)[left]{$E$};
\node at (F)[above]{$F$};
\node at (G)[above]{$G$};
\node at (H)[above]{$H$};
\node at (J2)[right]{$J$};

\node at (L)[left]{$E'$};
\node at (M)[right]{$F'$};
\node at (0.4,1.75)[above]{$G'$};
\node at (O)[right]{$H'$};

\draw [very thick] (E) -- (A) -- (T) -- (F) -- (E) -- (G) -- (H) ; 
\draw [very thick] (H) -- (J) --  (F)--(H) ;
\draw [very thick] (T) -- (J);
\draw [very thick] (J) -- (D) -- (T);

\draw (A) -- (C) --(G);
\draw (C) -- (D) --(H);

\draw [blue](K)--(L) ;
\draw[blue] (K)--(G);
\draw [blue](K) --(O) ;
\draw [blue](K) --(M) ;
\draw[dashed,blue] (N) --(G);
\draw [blue](K) -- (P) ;

\draw [red] (L)-- (R);
\draw [red] (S)-- (O);
\draw [dashed,red] (R)--(N)-- (S);
\draw[dashed,red] (N) --(C);
\draw [red] (L) -- (M)-- (O);
\draw [red] (L) --(A) ;
\draw [red] (T) -- (M);
\draw [red](O) -- (D) ;
\draw [red](D) -- (P) --(T);
\draw [red](M) -- (P) --(O);

\node at (K)[above]{$O$};

\end{tikzpicture}   
\bigskip
\caption {Planar representation according to  Cauchy. The pyramid with vertex 
$O$ is supposed translucent.} \label{rep-Cauchy}
\end{figure}

The hole formed by the removed face
 appears outside in the planar representation (see the blue part in figure {\ref{Cauchytriangule}}).

The next step of Cauchy's proof 
is to define a triangulation of $K$ 
by a subdivision of all the polygons. 
Notice that, in the triangulating process, 
the alternating sum $n_0 -n_1 + n_2$ does not change. 
The boundary of the hole 
 consists of edges with the following property: 
Each edge is a face of a triangle that has only this edge as the common edge with the hole.
For example, in figure \ref{Cauchyop} (a), the triangle $\sigma$ 
has the edge $\tau$ as the common edge with the hole. 
 The extension of the hole consists of two operations that we describe in what follows. 
 
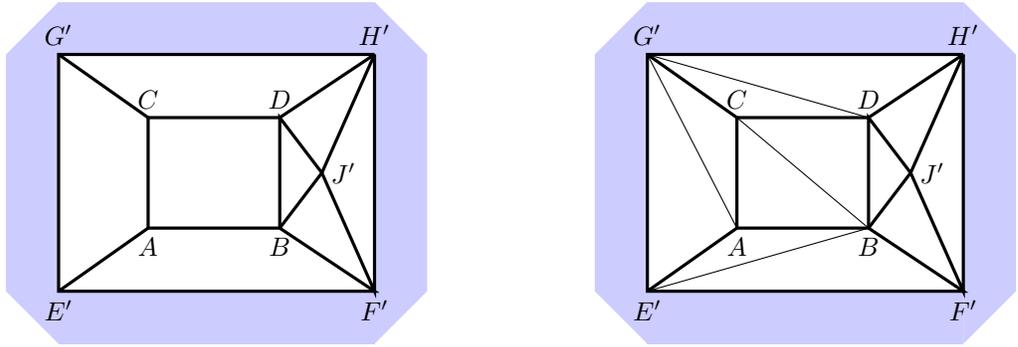
\begin{figure}[H]
\centering
\begin{tikzpicture}  [scale= 0.7] 
\coordinate (A) at (1.7,1.2);
\coordinate (T) at (4.2,1.2);
\coordinate  (C) at (1.7,3.3);
\coordinate  (D) at (4.2,3.3);
\coordinate (E) at (0,0);
\coordinate (F) at (6,0);
\coordinate  (G) at (0,4.5);
\coordinate  (H) at (6,4.5);
\coordinate (J) at (5,2.25);

\coordinate (L) at (0,-1);
\coordinate (M) at (-1,0);
\coordinate  (N) at (-1,4.5);
\coordinate  (O) at (0,5.5);
\coordinate (P) at (6,5.5);
\coordinate (Q) at (7,4.5);
\coordinate  (R) at (7,0);
\coordinate  (S) at (6,-1);

\fill[fill=blue!20] (L) -- (O) -- (N) --(M);
\fill[fill=blue!20] (P) -- (Q) -- (R) -- (S);
\fill[fill=blue!20] (O) -- (G) -- (H)-- (P);
\fill[fill=blue!20] (L) -- (E) -- (F)-- (S);

\node at (A)[below]{$A$};
\node at (T)[below]{$B$};
\node at (C)[above]{$C$};
\node at (D)[above]{$D$};
\node at (E)[below]{$E'$};
\node at (F)[below]{$F'$};
\node at (G)[above]{$G'$};
\node at (H)[above]{$H'$};
\node at (J)[right]{$J'$};

\draw [very thick] (E) -- (A) -- (T) -- (F) -- (E) -- (G) -- (H) ; 
\draw [very thick] (H) -- (J) --  (F)--(H) ;
\draw [very thick] (T) -- (J);
\draw [very thick] (J) -- (D) -- (T);

\draw [very thick] (A) -- (C) --(G);
\draw [very thick] (C) -- (D) --(H);

\end{tikzpicture}  
\qquad\qquad\qquad 
\begin{tikzpicture}  [scale= 0.7] 
\coordinate (A) at (1.7,1.2);
\coordinate (T) at (4.2,1.2);
\coordinate  (C) at (1.7,3.3);
\coordinate  (D) at (4.2,3.3);
\coordinate (E) at (0,0);
\coordinate (F) at (6,0);
\coordinate  (G) at (0,4.5);
\coordinate  (H) at (6,4.5);
\coordinate (J) at (5,2.25);

\coordinate (L) at (0,-1);
\coordinate (M) at (-1,0);
\coordinate  (N) at (-1,4.5);
\coordinate  (O) at (0,5.5);
\coordinate (P) at (6,5.5);
\coordinate (Q) at (7,4.5);
\coordinate  (R) at (7,0);
\coordinate  (S) at (6,-1);

\fill[fill=blue!20] (L) -- (O) -- (N) --(M);
\fill[fill=blue!20] (P) -- (Q) -- (R) -- (S);
\fill[fill=blue!20] (O) -- (G) -- (H)-- (P);
\fill[fill=blue!20] (L) -- (E) -- (F)-- (S);

\node at (A)[below]{$A$};
\node at (T)[below]{$B$};
\node at (C)[above]{$C$};
\node at (D)[above]{$D$};
\node at (E)[below]{$E'$};
\node at (F)[below]{$F'$};
\node at (G)[above]{$G'$};
\node at (H)[above]{$H'$};
\node at (J)[right]{$J'$};

\draw [very thick] (E) -- (A) -- (T) -- (F) -- (E) -- (G) -- (H) ; 
\draw [very thick] (H) -- (J) --  (F)--(H) ;
\draw [very thick] (T) -- (J);
\draw [very thick] (J) -- (D) -- (T);

\draw [very thick] (A) -- (C) --(G);
\draw [very thick] (C) -- (D) --(H);

\draw (G)--(D);
\draw (A) -- (G);
\draw (C)--(T)--(E);

\end{tikzpicture}   
\caption{The polygon $K$ and its triangulation. The hole is in blue.} \label{Cauchytriangule} 
\end{figure}

The first ``operation I'' consists of removing 
from the polyhedron $K$ such a triangle $\sigma$ together 
with its corresponding edge $\tau$ and then the hole is extended. 
 Operation I  does not change the sum $n_0 - n_1 + n_2$, because 
 $n_0$ does not change while $n_1$ and $n_2$  decrease by $1$. 

\begin{figure}[H]
\begin{tikzpicture}  [scale= 0.8] 
\coordinate (A) at (1.7,1.2);
\coordinate (T) at (4.2,1.2);
\coordinate  (C) at (1.7,3.3);
\coordinate  (D) at (4.2,3.3);
\coordinate (E) at (0,0);
\coordinate (F) at (6,0);
\coordinate  (G) at (0,4.5);
\coordinate  (H) at (6,4.5);
\coordinate (J) at (5,2.25);

\coordinate (L) at (0,-1);
\coordinate (M) at (-1,0);
\coordinate  (N) at (-1,4.5);
\coordinate  (O) at (0,5.5);
\coordinate (P) at (6,5.5);
\coordinate (Q) at (7,4.5);
\coordinate  (R) at (7,0);
\coordinate  (S) at (6,-1);

\fill[fill=blue!20] (L) -- (O) -- (N) --(M);
\fill[fill=blue!20] (P) -- (Q) -- (R) -- (S);
\fill[fill=blue!20] (O) -- (G) -- (H)-- (P);
\fill[fill=blue!20] (L) -- (E) -- (F)-- (S);

\draw [very thick] (E) -- (A) -- (T) -- (F) -- (E) -- (G) -- (H) ; 
\draw [very thick] (H) -- (J) --  (F) ;
\draw [very thick] (T) -- (J);
\draw [very thick] (J) -- (D) -- (T);

\draw [very thick] (A) -- (C) --(G);
\draw [very thick] (C) -- (D) --(H);

 \filldraw[fill=pink] (J)--(H)--(F);
\draw [very thick,red] (H) -- (F) ;
\node at (6.2,2.25){$\tau$};
\node at (5.55,2.5){$\sigma$};

\draw  [very thick] (G)--(D);
\draw  [very thick] (A) -- (G);
\draw  [very thick] (C)--(T)--(E);

\node at (4,-0.5)[right]{$(a)$};
\node at (0.5,-0.5)[right]{Operation I};
\end{tikzpicture}   
\qquad \qquad
\begin{tikzpicture}  [scale= 0.8] 
\coordinate (A) at (1.7,1.2);
\coordinate (T) at (4.2,1.2);
\coordinate  (C) at (1.7,3.3);
\coordinate  (D) at (4.2,3.3);
\coordinate (E) at (0,0);
\coordinate (F) at (6,0);
\coordinate  (G) at (0,4.5);
\coordinate  (H) at (6,4.5);
\coordinate (J) at (5,2.25);

\coordinate (L) at (0,-1);
\coordinate (M) at (-1,0);
\coordinate  (N) at (-1,4.5);
\coordinate  (O) at (0,5.5);
\coordinate (P) at (6,5.5);
\coordinate (Q) at (7,4.5);
\coordinate  (R) at (7,0);
\coordinate  (S) at (6,-1);

\fill[fill=blue!20] (L) -- (O) -- (N) --(M);
\fill[fill=blue!20] (P) -- (Q) -- (R) -- (S);
\fill[fill=blue!20] (O) -- (G) -- (H)-- (P);
\fill[fill=blue!20] (L) -- (E) -- (F)-- (S);

\draw [very thick] (E) -- (A) -- (T) -- (F) -- (E) -- (G) ; 
\draw [very thick] (J) --  (F) ;
\draw [very thick] (T) -- (J);
\draw [very thick] (J) -- (D) -- (T);

\draw [very thick] (A) -- (C) --(G);
\draw [very thick] (C) -- (D) ;

\fill[fill=blue!20] (F) -- (J) -- (H);
\fill[fill=blue!20] (G) -- (D) -- (H);
 \filldraw[fill=pink] (D)--(H)--(J);
\draw [very thick,red] (D) -- (H) -- (J) ;
\node at (4.9,4.1){$\tau_1$};
\node at (5.75,3.2){$\tau_2$};
\node at (5,3.3){$\sigma$};
\node at (H)[red] {$\bullet$};
\node at (H)[above]{$a$};

\draw  [very thick] (G)--(D);
\draw  [very thick] (A) -- (G);
\draw  [very thick] (C)--(T)--(E);

\node at (4,-0.5)[right]{$(b)$};
\node at (0.5,-0.5)[right]{Operation II};
\end{tikzpicture}   
\caption{The two Cauchy operations.}\label{Cauchyop}
\end{figure}
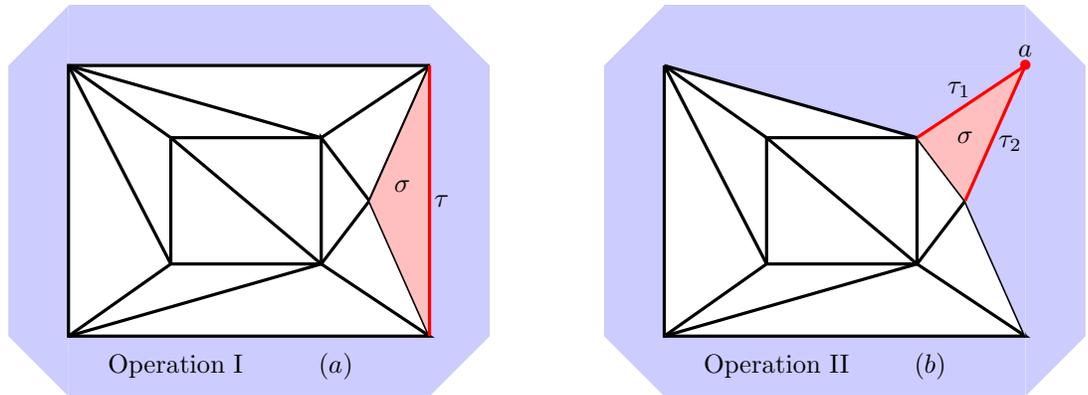

When a situation like the one of figure \ref{Cauchyop} (b) appears, 
where one triangle (here $\sigma$) has two common edges $\tau_1$ and $\tau_2$  
with the hole, 
we use ``operation II'' which consists  of 
removing from the polyhedron $K$ the triangle $\sigma$ together with the two edges  
$\tau_1$ and $\tau_2$ and the vertex $a$ that is the common vertex of  $\tau_1$ and  $\tau_2$. 
Then, the hole is extended. 

Operation II also does not change the sum $n_0 - n_1 + n_2$,
 since $n_0$ and $n_2$ decrease by 1 and  $n_1$  decreases by 2. 

If we take care of keeping the boundary homemorphic 
to a circle, then the hole is extended, using the two operations above, until we have only one triangle. 
In this triangle, we have $n_0 - n_1 + n_2 = 3 - 3 + 1 = 1$. 
Since we removed an (open) polygon at the beginning, we have  already $+1$ 
in the sum  $n_0 -n_1 + n_2$.  
Hence for any convex polyhedron, we have  
$n_0 - n_1 + n_2 = +2$. 
\end{preuve}

\subsection{After Cauchy} \label{depois_Cauchy}

Some authors, in particular Lakatos \cite{Lak},  
criticize Cauchy's proof.  In his book (see \cite{Lak}, pages 11 and 12),  Lakatos provided 
  a counter-example to Cauchy's process. Here, we adapt Lakatos' counter-example
 to our example in figure \ref{rep-Cauchy}.
Extending the hole by removing triangles in the indicated order in figure \ref{ordemLakatos} (a), 
 {we use} operations I and II of Cauchy until the {ninth} triangle. 
 If we remove the  {tenth} triangle, 
the hole disconnects the rest of the figure (see figure {\ref{ordemLakatos} (b)}): 
The eleventh and twelfth triangles 
 are no longer connected. 
Moreover, the boundary of the hole is no longer homeomorphic to a circle. 
Finally,  we observe that if  we remove the  {tenth} triangle, we 
do not remove any vertex, but we remove  two edges and {one} triangle, 
therefore the sum $n_0 - n_1 + n_2$ is not preserved. 

Hence, we need to be very careful concerning the order of the removal of the triangles since 
 it can happen that the hole disconnects the polyhedron $K$.  
Moreover, the boundary of the hole is no longer a curve homeomorphic to a circle because  
it has multiple points.

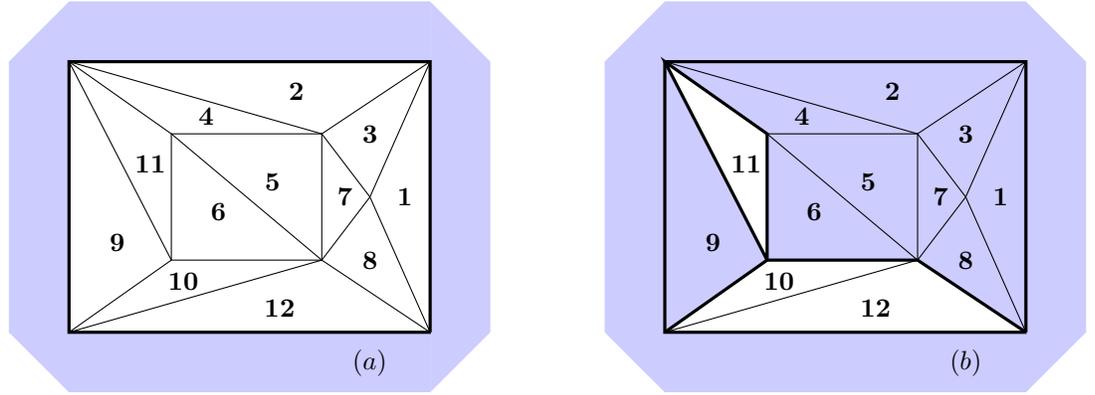
\begin{figure}[H]
\begin{tikzpicture}  [scale= 0.8] 
\coordinate (A) at (1.7,1.2);
\coordinate (T) at (4.2,1.2);
\coordinate  (C) at (1.7,3.3);
\coordinate  (D) at (4.2,3.3);
\coordinate (E) at (0,0);
\coordinate (F) at (6,0);
\coordinate  (G) at (0,4.5);
\coordinate  (H) at (6,4.5);
\coordinate (J) at (5,2.25);

\coordinate (L) at (0,-1);
\coordinate (M) at (-1,0);
\coordinate  (N) at (-1,4.5);
\coordinate  (O) at (0,5.5);
\coordinate (P) at (6,5.5);
\coordinate (Q) at (7,4.5);
\coordinate  (R) at (7,0);
\coordinate  (S) at (6,-1);

\fill[fill=blue!20] (L) -- (O) -- (N) --(M);
\fill[fill=blue!20] (P) -- (Q) -- (R) -- (S);
\fill[fill=blue!20] (O) -- (G) -- (H)-- (P);
\fill[fill=blue!20] (L) -- (E) -- (F)-- (S);

\draw (E) -- (A) -- (T) -- (F) ;
\draw[very thick] (F) -- (E) -- (G) -- (H)--(F) ; 
\draw (H) -- (J) --  (F)--(H) ;
\draw  (T) -- (J);
\draw  (J) -- (D) -- (T);

\draw  (A) -- (C) --(G);
\draw   (C) -- (D) --(H);

\draw (G)--(D);
\draw (A) -- (G);
\draw (C)--(T)--(E);

\node at (5.3,2.25)[right]{$\bf 1$};
\node at (3.5,4)[right]{$\bf 2$};
\node at (5,3.3){$\bf 3$};
\node at (2,3.6)[right]{$\bf 4$};
\node at (3.1,2.5)[right]{$\bf 5$};
\node at (2.2,2)[right]{$\bf 6$};
\node at (4.3,2.25)[right]{$\bf 7$};
\node at (5,1.2){$\bf 8$};
\node at (0.8,1.5){$\bf 9$};
\node at (1.9,0.85) {$\bf 10$};
\node at (1.35,2.8) {$\bf 11$};
\node at (3.5,0.4) {$\bf 12$};

\node at (5,-0.5) {$(a)$};
\end{tikzpicture}   
\qquad\qquad
\begin{tikzpicture}  [scale= 0.8] 
\coordinate (A) at (1.7,1.2);
\coordinate (T) at (4.2,1.2);
\coordinate  (C) at (1.7,3.3);
\coordinate  (D) at (4.2,3.3);
\coordinate (E) at (0,0);
\coordinate (F) at (6,0);
\coordinate  (G) at (0,4.5);
\coordinate  (H) at (6,4.5);
\coordinate (J) at (5,2.25);

\coordinate (L) at (0,-1);
\coordinate (M) at (-1,0);
\coordinate  (N) at (-1,4.5);
\coordinate  (O) at (0,5.5);
\coordinate (P) at (6,5.5);
\coordinate (Q) at (7,4.5);
\coordinate  (R) at (7,0);
\coordinate  (S) at (6,-1);

\fill[fill=blue!20] (G) -- (C) -- (D) --(H);
\fill[fill=blue!20] (T) -- (F) -- (H) -- (D);
\fill[fill=blue!20] (A) -- (T) -- (D)-- (C);
\fill[fill=blue!20] (G) -- (E) -- (A);

\fill[fill=blue!20] (L) -- (O) -- (N) --(M);
\fill[fill=blue!20] (P) -- (Q) -- (R) -- (S);
\fill[fill=blue!20] (O) -- (G) -- (H)-- (P);
\fill[fill=blue!20] (L) -- (E) -- (F)-- (S);

 
 \node at (5.3,2.25)[right]{$\bf 1$};
\node at (3.5,4)[right]{$\bf 2$};
\node at (5,3.3){$\bf 3$};
\node at (2,3.6)[right]{$\bf 4$};
\node at (3.1,2.5)[right]{$\bf 5$};
\node at (2.2,2)[right]{$\bf 6$};
\node at (4.3,2.25)[right]{$\bf 7$};
\node at (5,1.2){$\bf 8$};
\node at (0.8,1.5){$\bf 9$};
\node at (1.9,0.85) {$\bf 10$};
\node at (1.35,2.8) {$\bf 11$};
\node at (3.5,0.4) {$\bf 12$};

\draw[very thick] (E)--(A) -- (T) -- (F);
\draw (E)--(T);
\draw[very thick] (F) -- (E) -- (G) -- (H)--(F) ; 
\draw (H) -- (J) --  (F)--(H) ;
\draw  (T) -- (J);
\draw  (J) -- (D) -- (T);
\draw (G)--(D);
\draw [very thick]  (A) -- (C) --(G)--(A);
\draw   (C) -- (D) --(H);
\draw (C)--(T);

\node at (5,-0.5) {$(b)$};
\end{tikzpicture}   
\bigskip
\caption{{Removal} order according to Lakatos.} \label{ordemLakatos}
\end{figure}

In the paper \cite{Li2}, Lima formalized the arguments of Lakatos and described the situation of figures  \ref{figuraCE}. 
In figures (a), (b) and (c),  the extension of the hole, 
 obtained by the removal of the triangle $\sigma$ from the polyhedron $K$, 
on the one hand, changes the sum  $n_0 - n_1 + n_2$ 
and, on the other hand, disconnects the polyhedron $K$. 
 Lakatos' example corresponds to the situation (a) in figure of Lima. 
In the examples of Lima, the boundary of the hole 
is a curve with multiple points. 

We observe that figure (d) is also a case where the boundary is a curve having multiple points. 
However, the sum $n_0 - n_1 + n_2$ is preserved when we 
remove the triangle $\sigma$ from $K$ since 
we remove two vertices $x_2$ and $x_3$, 
three edges $(x_1,x_2)$, $(x_2,x_3)$, $(x_1,x_3)$ and the triangle $\sigma$. 
It seems that Lima did not realize 
that this case is admissible.  
The situation of figure (d) can be also used in the process of Cauchy. 
In this paper, we call this operation ``operation III''. 
 This operation appeared also in \cite{Cel}.  

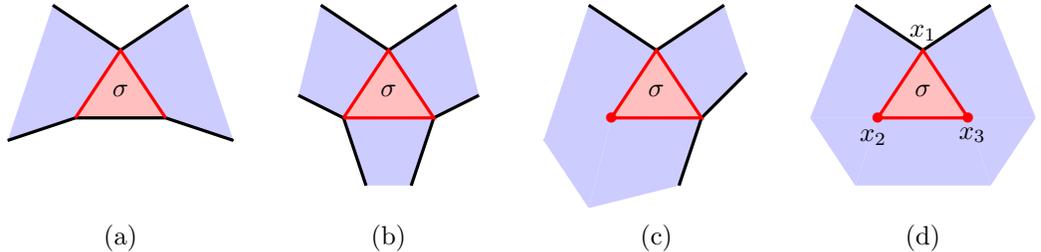
\begin{figure} [H] 
\begin{tikzpicture}  [scale= 0.30]

\coordinate (A) at (-2,0);
\coordinate (B) at (2,0);
\coordinate  (C) at (0,3);
\coordinate  (D) at (-3,5);
\coordinate (E) at (3,5);
\coordinate (F) at (-5, -1);
\coordinate  (G) at (5,-1);
\coordinate  (H) at (-4,1);
\coordinate  (I) at (4,1);
\coordinate (J) at (-1,-3);
\coordinate (K) at (1,-3);
\coordinate (M) at (-3, -4);
\coordinate  (N) at (1,-3);
\coordinate  (0) at (4,2);
\coordinate  (P) at (-5,0);
\coordinate (Q) at (-3,-3);
\coordinate (R) at (3,-3);
\coordinate (S) at (5,0);

  \draw (E)   --  (C) -- (B) -- (G) ; 

   \filldraw[fill=blue!20] (D)   --  (C) -- (A) -- (F) ; 
     \filldraw[fill=blue!20] (E)   --  (C) -- (B) -- (G) ; 
     
     \filldraw[fill=pink] (A)   --  (C) -- (B) ; 
     
      \draw [very thick](D)   --  (C) -- (E) ; 
        \draw[very thick] (F)   --  (A) -- (B) -- (G) ; 
       \draw[very thick,red]  (A)   --  (C)-- (B) ; 
         
           \node (c) at (0,1.2) {$\sigma$}; 
              \node (a) at (0,-6.3) [above] {$\rm (a)$}; 
   \end{tikzpicture}
\qquad 
    \begin{tikzpicture}  [scale= 0.30] 

\coordinate (A) at (-2,0);
\coordinate (B) at (2,0);
\coordinate  (C) at (0,3);
\coordinate  (D) at (-3,5);
\coordinate (E) at (3,5);
\coordinate (F) at (-5, -1);
\coordinate  (G) at (5,-1);
\coordinate  (H) at (-4,1);
\coordinate  (I) at (4,1);
\coordinate (J) at (-1,-3);
\coordinate (K) at (1,-3);
\coordinate (M) at (-3, -4);
\coordinate  (N) at (1,-3);
\coordinate  (0) at (4,2);
\coordinate  (P) at (-5,0);
\coordinate (Q) at (-3,-3);
\coordinate (R) at (3,-3);
\coordinate (S) at (5,0);

   \filldraw[fill=blue!20] (D)   --  (C) -- (A) -- (H) ; 
     \filldraw[fill=blue!20] (E)   --  (C) -- (B) -- (I) ; 
      \filldraw[fill=blue!20] (J) -- (A)   --  (B)-- (K) ; 

     \filldraw[fill=pink] (A)   --  (C) -- (B) ; 
      
         \draw [very thick](D)   --  (C) -- (E) ; 
        \draw[very thick] (H)   --  (A) -- (J) ; 
           \draw[very thick] (K)   --  (B) -- (I); 
       \draw[very thick,red]  (A)   --  (C)-- (B)--(A) ; 
   
     \node (c) at (0,1.2) {$\sigma$}; 
 \node (a) at (0,-6.3) [above] {$\rm (b)$}; 
   \end{tikzpicture}
\qquad 
\begin{tikzpicture}  [scale= 0.30] 

\coordinate (A) at (-2,0);
\coordinate (B) at (2,0);
\coordinate  (C) at (0,3);
\coordinate  (D) at (-3,5);
\coordinate (E) at (3,5);
\coordinate (F) at (-5, -1);
\coordinate  (G) at (5,-1);
\coordinate  (H) at (-4,1);
\coordinate  (I) at (4,1);
\coordinate (J) at (-1,-3);
\coordinate (K) at (1,-3);
\coordinate (M) at (-3, -4);
\coordinate  (N) at (1,-3);
\coordinate  (O) at (4,2);
\coordinate  (P) at (-5,0);
\coordinate (Q) at (-3,-3);
\coordinate (R) at (3,-3);
\coordinate (S) at (5,0);

   \fill[fill=blue!20] (D)   --  (C) -- (A) -- (F) ; 
     \fill[fill=blue!20] (F)   --  (A) -- (M)  ; 
        \fill[fill=blue!20] (M)   -- (A) --(B) -- (N) ; 
   \filldraw[fill=blue!20] (E)   --  (C) -- (B) -- (O) ; 
   \draw[very thick,blue!20] (F)--(A)--(H);
   
     \filldraw[fill=pink] (A)   --  (C) -- (B) ; 
              
                    \draw [very thick](D)   --  (C) -- (E) ; 
           \draw[very thick] (K)   --  (B) -- (O); 
       \draw[very thick,red]  (A)   --  (C)-- (B)--(A) ; 
            
            \node (u) at (-2,0) {\textcolor{red}{$\bullet$}}; 
              \node (c) at (0,1.2) {$\sigma$}; 
             \node (a) at (0,-6.3) [above] {$\rm (c)$}; 
   \end{tikzpicture}
\qquad 
\begin{tikzpicture}  [scale= 0.30]

\coordinate (A) at (-2,0);
\coordinate (B) at (2,0);
\coordinate  (C) at (0,3);
\coordinate  (D) at (-3,5);
\coordinate (E) at (3,5);
\coordinate (F) at (-5, -1);
\coordinate  (G) at (5,-1);
\coordinate  (H) at (-4,1);
\coordinate  (I) at (4,1);
\coordinate (J) at (-1,-3);
\coordinate (K) at (1,-3);
\coordinate (M) at (-3, -4);
\coordinate  (N) at (1,-3);
\coordinate  (0) at (4,2);
\coordinate  (P) at (-5,0);
\coordinate (Q) at (-3,-3);
\coordinate (R) at (3,-3);
\coordinate (S) at (5,0);

   \fill[fill=blue!20] (D)   --  (C) -- (A) -- (P) ; 
     \fill[fill=blue!20] (P)   --  (A) -- (Q)  ; 
        \fill[fill=blue!20] (Q)   -- (A) --(B) -- (R) ; 
        \fill[fill=blue!20] (R)   --  (B) -- (S)  ;    
   \fill[fill=blue!20] (S)   --  (B) -- (C) -- (E) ; 
   
     \filldraw[fill=pink] (A)   --  (C) -- (B) ; 
        
             \draw [very thick](D)   --  (C) -- (E) ; 
       \draw[very thick,red]  (A)   --  (C)-- (B)--(A) ; 
               \node (v) at (2,0) {\textcolor{red}{$\bullet$}}; 
            \node (u) at (-2,0) {\textcolor{red}{$\bullet$}}; 
            
      \node (a) at (0,3) [above] {$x_1$}; 
      \node (b) at (-2.2,-0.1) [below]{$x_2$}; 
      \node (c) at (2.2,0)[below] {$x_3$}; 
         \node (c) at (0,1.2) {$\sigma$}; 
          \node (a) at (0,-6.3) [above] {{$\rm (d)$}}; 
   \end{tikzpicture}

\caption{The counter-examples of Lima: 
 The hole $B$ is in blue and the simplices to be removed are in 
 red and pink.} 
\label{figuraCE}
\end{figure}

Note that with operation III, Lakatos' example is no longer a counter-example: 
come back to figure \ref{ordemLakatos} (a), after removing the ninth triangle, one can remove the twelfth triangle by operation II, then remove the tenth triangle by operation III. We have only  the eleventh triangle for which $n_0 - n_1 + n_2 = +1$ so we are done!

Some authors suggest a strategy to define an order of removal of the triangles that allows the use of 
{Cauchy's method} to obtain 
the result. For example, Kirk \cite{Kirk} 
suggests the following strategy:
 ``{\it There are two important rules to follow when doing this. Firstly, we must always perform 
 {\rm [{Operation II}]} when {it is} possible to do so; if there is a choice between  {\rm [{Operation I}]} 
 and  {\rm [{Operation II}]} we must always choose  {\rm [{Operation II}]}. If we do not, the network 
 may break up into separate pieces. Secondly, we must only remove faces one at a time.}'' 
 We provide, in figure \ref{contra-exemplo-Abigail} (b), an example whose the process follows the 
 rules of the strategy defined by Kirk but the process disconnects the polyhedron. 
 So these rules are not sufficient. It is easy to build an example showing that they are not necessary.

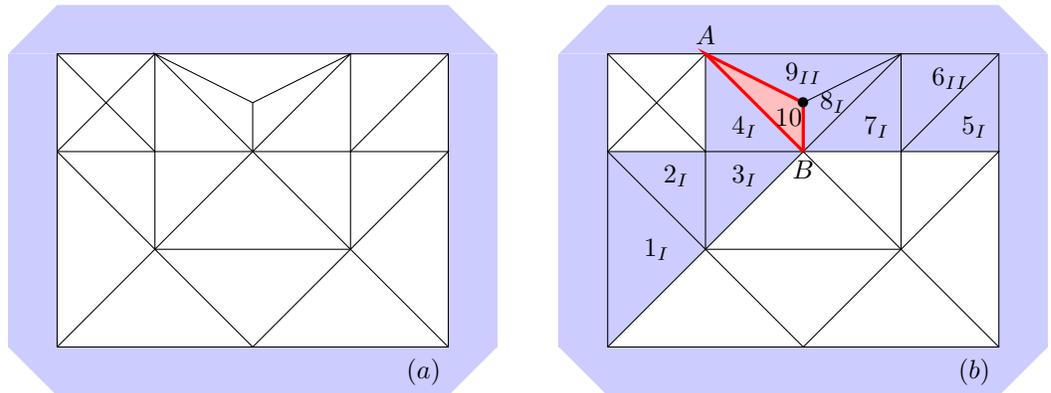
\begin{figure} [H]
\begin{tikzpicture}  [scale= 0.65]

\coordinate (A) at (0,0);
\coordinate (B) at (4,0);
\coordinate  (C) at (8,0);
\coordinate  (D) at (2,2);
\coordinate (E) at (6,2);
\coordinate (F) at (0,4);
\coordinate  (G) at (2,4);
\coordinate  (H) at (4,4);
\coordinate (J) at (6,4);
\coordinate (K) at (8,4);
\coordinate  (L) at (1,5);
\coordinate (M) at (4,5);
\coordinate (N) at (0,6);
\coordinate  (O) at (7,6);
\coordinate  (P) at (2,6);
\coordinate  (Q) at (6,6);
\coordinate (R) at (8,6);

\coordinate (S) at (-1,6);
\coordinate  (T) at (0,7);
\coordinate (U) at (8,7);
\coordinate (V) at (9,6);
\coordinate  (W) at (-1,0);
\coordinate  (X) at (0,-1);
\coordinate  (Y) at (8,-1);
\coordinate (Z) at (9,0);

\fill[fill=blue!20] (S) -- (T) -- (U) --(V);
\fill[fill=blue!20] (W) -- (X) -- (Y)--(Z);
\fill[fill=blue!20] (S) -- (N) -- (A)--(W);
\fill[fill=blue!20] (C) -- (R) -- (V)--(Z);

\draw (A)--(C)--(R)--(N)--(A);
\draw   (A)--(D)--(B)--(E)--(C);
\draw   (F)--(D)--(E)--(K)--(J)--(E);
\draw (E)--(H) -- (D) -- (G) -- (F);
\draw (F) -- (L) -- (G) --(H) -- (J) -- (R);
\draw (J) -- (Q) -- (H)--(M) --(Q) ;
\draw (H) -- (P) --  (M)  ;
\draw (G) -- (P) --  (L)--(N)  ;

\node at (7.5,-0.5) {$(a)$};
\end{tikzpicture}   
\qquad
\begin{tikzpicture}  [scale= 0.65] 

\coordinate (A) at (0,0);
\coordinate (B) at (4,0);
\coordinate  (C) at (8,0);
\coordinate  (D) at (2,2);
\coordinate (E) at (6,2);
\coordinate (F) at (0,4);
\coordinate  (G) at (2,4);
\coordinate  (H) at (4,4);
\coordinate (J) at (6,4);
\coordinate (K) at (8,4);
\coordinate  (L) at (1,5);
\coordinate (M) at (4,5);
\coordinate (N) at (0,6);
\coordinate  (O) at (7,6);
\coordinate  (P) at (2,6);
\coordinate  (Q) at (6,6);
\coordinate (R) at (8,6);

\coordinate (S) at (-1,6);
\coordinate  (T) at (0,7);
\coordinate (U) at (8,7);
\coordinate (V) at (9,6);
\coordinate  (W) at (-1,0);
\coordinate  (X) at (0,-1);
\coordinate  (Y) at (8,-1);
\coordinate (Z) at (9,0);

\fill[fill=blue!20] (S) -- (T) -- (U) --(V);
\fill[fill=blue!20] (W) -- (X) -- (Y)--(Z);
\fill[fill=blue!20] (S) -- (N) -- (A)--(W);
\fill[fill=blue!20] (C) -- (R) -- (V)--(Z);

\fill[fill=blue!20] (K) -- (R) -- (Q) --(J);
\fill[fill=blue!20] (A) -- (D) -- (H) -- (P) --(G) -- (F) --(A);
\fill[fill=blue!20] (Q) --(J)--(H)--(M)--(Q);
\fill[fill=blue!20] (Q) --(J)--(H)--(M)--(Q);
\fill[fill=blue!20] (M)--(P) -- (Q);
\fill[fill=pink] (M)--(P) -- (H);

\draw (A)--(C)--(R)--(N)--(A);
\draw   (A)--(D)--(B)--(E)--(C);
\draw   (F)--(D)--(E)--(K)--(J)--(E);
\draw (E)--(H) -- (D) -- (G) -- (F);
\draw (F) -- (L) -- (G) --(H) -- (J) -- (R);
\draw (J) -- (Q) -- (H)--(M) --(Q) ;
\draw (H) -- (P) --  (M)  ;
\draw (G) -- (P) --  (L)--(N)  ;

\draw[very thick,red] (H)--(M) -- (P)--(H);
\node at (M) {$\bullet$};

\node at (7.5,4.5){$5_I$};
\node at (7,5.5){$6_{II}$};
\node at (5.5,4.5){$7_I$};
\node at (1,2){$1_{I}$};
\node at (1.4,3.5){$2_I$};
\node at (2.8,3.5){$3_{I}$};
\node at (2.8,4.5){$4_I$};
\node at (4.6,5){$8_I$};
\node at (4,5.6){$9_{II}$};
\node at (3.7,4.7){$10$};

\node at (P) [above]{$A$};
\node at (H) [below]{$B$};
\node at (7.5,-0.5) {$(b)$};
\end{tikzpicture}    

 \medskip 
 \caption{A {counter-example} to the ``strategy'' suggested by Kirk. 
Triangles are removed in the order of numbers. The index I or II below each number
means that the triangle is removed using 
the corresponding operation I or II,  respectively. After removing the seventh triangle, 
the boundary of the hole is no longer homeomorphic to a circle. }  \label{contra-exemplo-Abigail} 
\end{figure}

The example provided in figure \ref{contra-exemplo-Abigail} is also a counter-example 
showing that even with operation III, Cauchy's process does not work 
with the given order of removal.
In fact, after removing the seventh triangle, we can continue untill the tenth triangle, 
but we cannot remove it because in that case we remove one vertex, three segments and one triangle. 
Notice that the vertices $A$ and $B$ do not belong to the hole. 

After Cauchy, many authors proposed alternating proofs 
of Euler's formula (\ref{Euler}), using original arguments. 
See the site of Eppstein \cite{Epp} 
containing 20 different proofs, 
using tools that appeared only after Cauchy's time. 
In particular, some proofs use the Jordan Lemma  (Jordan \cite{Jo}, 1866). 
 In fact, we will see that Jordan curves will appear in 
the proof of our theorem \ref{teo1} as an artifact.

However, to our knowledge, no one 
has given a strategy for removing the triangles that allows 
only the use of Cauchy's 
method and tools known in Cauchy's time. That is  the goal of our paper.

\subsection{Generalization of the {formula} (\ref{Euler}).}\label{genera}

The formula (\ref{Euler}) was extended  by many authors, in particular by  
Lhuilier, first for orientable surfaces of genus  $g$, as follows:
\begin{equation}\label{Lhuilier}
 n_0 -n_1 + n_2 = 2-2g.
\end{equation}

In the non-orientable case,  the  {formula} is given by (see \cite{Mas}):
$$ n_0 -n_1 + n_2 = 2-g.$$

The general result was provided by Poincar\'e \cite{Po3,Po1} who proved that, for any triangulation of a polyhedron $X$ of  {dimension} $k$, 
 where $n_i$ is the number 
of simplices of dimension $i$, the sum 
\begin{equation}\label{Euler-Poincare}
\chi (X) = \sum_{i=0}^k (-1)^i n_i
\tag{3}
\end{equation}
 does not depend on the triangulation of $X$. 
This sum is called the Euler-Poincar\'e characteristic of  $X$.

Due to the  dimension
convention for polyhedra in section \ref{before}, 
``Euler's formula'' would be better written 
as the Euler-Poincar\'e characteristic of the convex 3-dimensional 
polyhedral in the form 
$$  n_0 -n_1 + n_2 - n_3 = +1.$$
Of course, $n_3$ is $+1$ anyway, but this 
form of Euler's formula seems more suitable.

\section{Cauchy's method in the proof of Euler's formula} \label{OTeo} 

The classical planar {representations} of surfaces such as 
the sphere, the torus, the projective plane and the Klein bottle
are examples of the following situation: The surface is homeomorphic to 
the geometric representation of a polygon $K$, itself homeomorphic to a disc $D$, 
such that there are possible identifications of simplices on the boundary $K_0$. 

\begin{figure}[H]
\begin{tikzpicture}[scale=1.0]

\draw[->, very thick] (-3.5,  4.5) -- (-3.75, 5); 
\draw [very thick] (-3.75, 5) -- (-4.25, 6);
\draw [->, very thick] (-4.5, 6.5) -- (-4.25, 6);
\draw [->, very thick] (-4.5, 6.5) -- (-4, 6.5);
\draw [very thick] (-4, 6.5) -- (-3, 6.5);
\draw [->, very thick] (-2.5, 6.5) -- (-3, 6.5);
\draw [->, very thick] (-2.5,6.5) -- (-2.75, 6);
\draw [very thick] (-2.75, 6) -- (-3.25, 5);
\draw [->, very thick] (-3.5, 4.5) -- (-3.25, 5);
\draw[very thick] (-4,5.5) -- (-3,5.5) -- (-3.5,6.5) -- (-4, 5.5); 

\node at (-3.75, 5) [left] {$a$}; 
\node at (-3.25, 5) [right] {$b$}; 
\node at (-4.25, 6) [left] {$a$}; 
\node at (-2.75, 6) [right] {$b$}; 
\node at (-3, 6.6) [above] {$c$};
\node at (-4, 6.5) [above] {$c$};

\draw[->, very thick] (0,  5) -- (0.5, 5); 
\draw[ ->, very thick] (0.5, 5) -- (2.5, 5); 
\draw[->, very thick] (2.5,  5) -- (3.5, 5); 
\draw[very thick] (3.5,  5) -- (4, 5); 
\node at (0.5, 5) [below] {$e$}; 
\node at (2.5, 5) [below] {$g$}; 
\node at (3.5, 5) [below] {$f$}; 

\draw[->, very thick] (0,  6) -- (0.5, 6); 
\draw[ ->, very thick] (0.5, 6) -- (2.5, 6); 
\draw[->, very thick] (2.5,  6) -- (3.5, 6); 
\draw[very thick] (3.5,  6) -- (4, 6); 
\node at (0.5, 6) [above] {$c$}; 
\node at (2.5, 6) [above] {$a$}; 
\node at (3.5, 6) [above] {$b$}; 

\draw[->, very thick] (0, 5) -- (0, 5.5); 
\node at (0, 5.5) [left] {$d$}; 
\draw[ very thick] (0, 5.5) -- (0,6); 

\draw[->, very thick] (4, 5) -- (4, 5.5); 
\node at (4, 5.5) [right] {$d$}; 
\draw[ very thick] (4, 5.5) -- (4,6); 

\draw[very thick] (3, 5) -- (3,6); 

\draw[->, very thick] (1,  4) -- (1, 4.5); 
\draw[very thick] (1, 4.5) -- (1, 6.5); 
\draw[->, very thick] (1, 7) -- (1, 6.5); 
\draw[-<, very thick] (1, 7) -- (1.55, 7); 
\node at (1, 4.5) [left] {$e$}; 
\node at (1, 6.5) [left] {$c$}; 
\node at (1.6, 3.8) [left] {$f$}; 

\draw[very thick] (2,  4) -- (2, 4.5); 
\draw[->, very thick] (2, 6) -- (2, 4.5); 
\draw[->, very thick] (2, 6) -- (2, 6.5); 
\draw[very thick] (2, 7) -- (2, 6.5); 
\draw[-<, very thick] (1, 4) -- (1.55, 4); 
\node at (2, 4.5) [left] {$g$}; 
\node at (2, 6.5) [left] {$a$}; 
\node at (1.8, 7.25) [left] {$b$}; 

\draw[very thick] (1, 4) -- (2,4) ; 
\draw[very thick] (1, 7) -- (2,7) ; 

\draw[very thick] (0, 5) -- (2, 7);
\draw[very thick] (1, 5) -- (2, 6);
\draw[very thick] (1, 4) -- (3, 6);
\draw[very thick] (3, 5) -- (4, 6);

\end{tikzpicture}
\caption{Planar representations of the sphere.}\label{triangsphere2}
\end{figure}
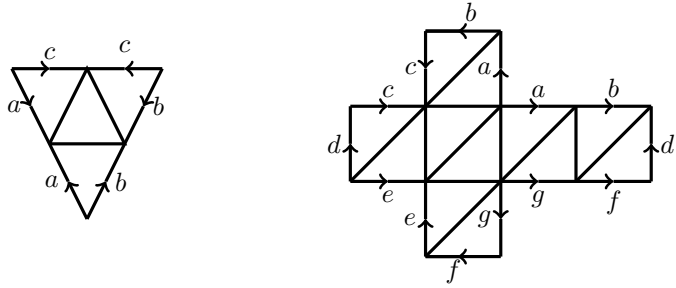

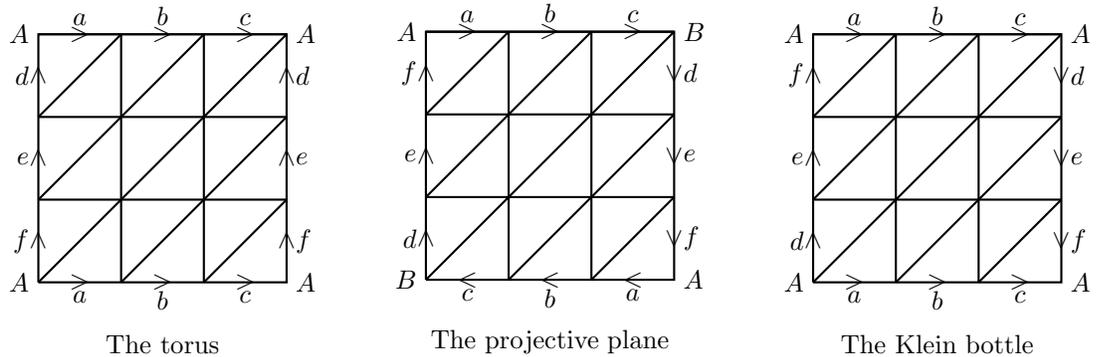
\begin{figure}[H]
\begin{tikzpicture}[scale=0.55]

 \draw [thick] (0, 0) -- (6,0) -- (6,6) -- (0,6) -- (0,0); 

\draw [thick] (2, 0) -- (2, 6); 
\draw [thick] (4, 0) -- (4, 6); 
\draw [thick](0, 2) -- (6, 2);
\draw [thick] (0, 4) -- (6, 4); 
\draw [thick] (0, 6) -- (6, 6); 

\draw [thick] (2,6) -- (0,4); 
\draw [thick] (4,6) -- (0,2); 
\draw [thick] (6,6) -- (0,0); 
\draw [thick] (6,4) -- (2,0); 
\draw [thick] (6,2) -- (4,0); 

\node at (1,0){$>$};
\node at (1,0)[below]{$a$};
\node at (3,0){$>$};
\node at (3,0)[below]{$b$};
\node at (5,0){$>$};
\node at (5,0)[below]{$c$};

\node at (1,6){$>$};
\node at (1,6)[above]{$a$};
\node at (3,6){$>$};
\node at (3,6)[above]{$b$};
\node at (5,6){$>$};
\node at (5,6)[above]{$c$};

\node at (0, 1){$\wedge$};
\node at (0,1)[left]{$f$};
\node at (0, 3){$\wedge$};
\node at (0,3)[left]{$e$};
\node at (0, 5){$\wedge$};
\node at (0,5)[left]{$d$};

\node at (6, 1){$\wedge$};
\node at (6,1)[right]{$f$};
\node at (6, 3){$\wedge$};
\node at (6,3)[right]{$e$};
\node at (6,5){$\wedge$};
\node at (6,5)[right]{$d$};

\node at (0, 0)[left]{$A$};
\node at (6, 0)[right]{$A$};
\node at (0, 6)[left]{$A$};
\node at (6, 6)[right]{$A$};

\node at (3,-1.5) {The torus};
\end{tikzpicture}
\qquad 
\begin{tikzpicture}[scale=0.55]

 \draw [thick] (0, 0) -- (6,0) -- (6,6) -- (0,6) -- (0,0); 

\draw [thick] (2, 0) -- (2, 6); 
\draw [thick] (4, 0) -- (4, 6); 
\draw [thick](0, 2) -- (6, 2);
\draw [thick] (0, 4) -- (6, 4); 
\draw [thick] (0, 6) -- (6, 6); 

\draw [thick] (2,6) -- (0,4); 
\draw [thick] (4,6) -- (0,2); 
\draw [thick] (6,6) -- (0,0); 
\draw [thick] (6,4) -- (2,0); 
\draw [thick] (6,2) -- (4,0); 

\node at (1,0){$<$};
\node at (1,0)[below]{$c$};
\node at (3,0){$<$};
\node at (3,0)[below]{$b$};
\node at (5,0){$<$};
\node at (5,0)[below]{$a$};

\node at (1,6){$>$};
\node at (1,6)[above]{$a$};
\node at (3,6){$>$};
\node at (3,6)[above]{$b$};
\node at (5,6){$>$};
\node at (5,6)[above]{$c$};

\node at (0, 1){$\wedge$};
\node at (0,1)[left]{$d$};
\node at (0, 3){$\wedge$};
\node at (0,3)[left]{$e$};
\node at (0, 5){$\wedge$};
\node at (0,5)[left]{$f$};

\node at (6, 1){$\vee$};
\node at (6,1)[right]{$f$};
\node at (6, 3){$\vee$};
\node at (6,3)[right]{$e$};
\node at (6,5){$\vee$};
\node at (6,5)[right]{$d$};

\node at (0, 0)[left]{$B$};
\node at (6, 0)[right]{$A$};
\node at (0, 6)[left]{$A$};
\node at (6, 6)[right]{$B$};

\node at (3,-1.5) {The projective plane};

\end{tikzpicture}
\qquad
\begin{tikzpicture}[scale=0.55]
 \draw[thick]  (0, 0) -- (6,0) -- (6,6) -- (0,6) -- (0,0); 

\draw [thick] (2, 0) -- (2, 6); 
\draw [thick] (4, 0) -- (4, 6); 
\draw [thick](0, 2) -- (6, 2);
\draw [thick] (0, 4) -- (6, 4); 
\draw [thick] (0, 6) -- (6, 6); 

\draw [thick] (2,6) -- (0,4); 
\draw [thick] (4,6) -- (0,2); 
\draw [thick] (6,6) -- (0,0); 
\draw [thick] (6,4) -- (2,0); 
\draw [thick] (6,2) -- (4,0); 

\node at (1,0){$>$};
\node at (1,0)[below]{$a$};
\node at (3,0){$>$};
\node at (3,0)[below]{$b$};
\node at (5,0){$>$};
\node at (5,0)[below]{$c$};

\node at (1,6){$>$};
\node at (1,6)[above]{$a$};
\node at (3,6){$>$};
\node at (3,6)[above]{$b$};
\node at (5,6){$>$};
\node at (5,6)[above]{$c$};

\node at (0, 1){$\wedge$};
\node at (0,1)[left]{$d$};
\node at (0, 3){$\wedge$};
\node at (0,3)[left]{$e$};
\node at (0, 5){$\wedge$};
\node at (0,5)[left]{$f$};

\node at (6, 1){$\vee$};
\node at (6,1)[right]{$f$};
\node at (6, 3){$\vee$};
\node at (6,3)[right]{$e$};
\node at (6,5){$\vee$};
\node at (6,5)[right]{$d$};

\node at (0, 0)[left]{$A$};
\node at (6, 0)[right]{$A$};
\node at (0, 6)[left]{$A$};
\node at (6, 6)[right]{$A$};

\node at (3,-1.5) {The Klein bottle};

\end{tikzpicture}
\caption{Planar representations of the torus, the projective plane and the  Klein bottle.} \label{Klein}
\end{figure}

In this section, we prove the following theorem using only Cauchy's method  
({section} \ref{methode_Cauchy}) and sub-triangulations:
\begin{theorem} \label{teo1}
Let $K$ be a triangulated polygon in $\R^2$, 
homeomorphic to a disc $D$, 
with possible identifications of simplices on the boundary $K_0$ of $K$. 
 We have
$$\chi(K) = \chi (K_0) +1.$$
\end{theorem}

We emphasize that the important point of our proof is that, 
from the given triangulation, 
we provide a sub-triangulation such that we can prove the theorem using only 
 Cauchy's method 
({section} \ref{methode_Cauchy}), without using other tools. 

\begin{proof}[\bf{Proof of theorem \ref{teo1} using only Cauchy's method}] 
Given a polygon $K$  triangulated and homeomorphic to a disc in $\R^2$ 
with possible identifications of the simplices on the boundary  $K_0$ of $K$, 
the proof consists of six steps, as follows.

\medskip 

1) Step 1:  The first step is to construct a ``lifting'' of $K$ into  a pyramidal shape. 
Here, we call the ``pyramid'' only the surface (dimension 2) of the pyramid, 
{\it i.e.} the union of the faces of the pyramid without the base. 

We can assume that the origin $0$ of $\R^2$ lies inside a 2-dimensional simplex $\sigma_0$  
in the interior of the polygon $K$.


   \begin{figure}[H] 
\begin{tikzpicture}  [scale= 0.5]

\coordinate [label=left:$b_1$] (B1) at (-6,-0.5);
\coordinate [label=below:$b_6$] (B2) at (-3.5,-4);
\coordinate [label=right:$b_5$] (B3) at (4.5,-4);
\coordinate [label=right:$b_4$] (B4) at (6,-1);
\coordinate [label=right:$b_3$] (B5) at (3.5,4);
\coordinate [label=above:$b_2$] (B6) at (-3,4);
\draw [very thick] (B1)  -- (B2) -- (B3) -- (B4)--(B5) -- (B6) -- (B1);

\coordinate [label=right:$a_2$]  (A1) at (0,2);
\coordinate [label=left:$a_1$] (A2) at (-1,0);
\coordinate [label=below:$a_3$] (A3) at (2,0);
\draw [very thick] (A1)  -- (A2) -- (A3)-- (A1) ;
\coordinate [label=right:$0$]  (0) at (0.3,0.7);
 \foreach \point in {0}
    \fill [black,opacity=.5] (\point) circle (2pt);

\coordinate [label=below:$y_2$] (X1) at (0.8,-2);
\coordinate  [label=left:$y_1$] (X2) at (-2,1);

\draw (B1) -- (X2);
\draw (A2) -- (X2);
\draw (A1) -- (X2);
\draw (B6) -- (X2);
\draw (B2) -- (X2);
\draw (B6) -- (A1);

\draw (A2) -- (X1);
\draw (A3) -- (X1);
\draw (B2) -- (X1);
\draw (B3) -- (X1);
\draw (B3) -- (A3);
\draw (B4) -- (A3);
\draw (B5) -- (A3);
\draw (B5) -- (A1);
\draw (B2) -- (A2);

 \end{tikzpicture}
\caption{The triangulation  of the polygon $K$.}\label{Pyramide70}
\end{figure}
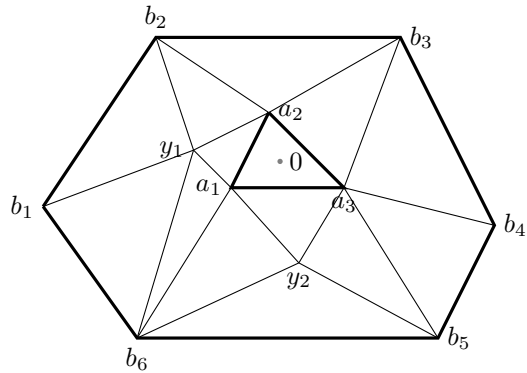

Let us consider the Euclidean metric  in $\R^2$.  
We can assume that the distances from the origin to the vertices of the triangulation are  different, 
{otherwise a small perturbation will not change}  the structure of the 
 simplicial  complex and the proof of the theorem can be processed in the same way.

Let us denote by $a_1, a_2, a_3$ the vertices of $\sigma_0$ 
and by $b_1, \ldots , b_k$ the vertices of the triangulation of $K_0$. 
The other vertices are denoted by the following way: 
We call $y_1$ the vertex nearest to the origin $0$ 
 and $y_{2}, \ldots, y_{n}$ the vertices in the increasing order of distances from $0$. 

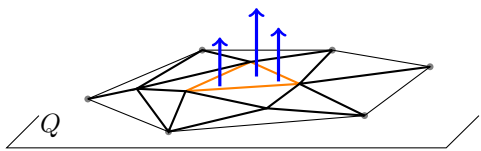
\begin{figure}[H] 
\begin{tikzpicture}  [scale= 0.65]
\draw (-3.76,-3.2) node {$Q$};
\draw (-4,-3) --(-4.66, -3.66) -- (4.33,-3.66) -- (5,-3);
\coordinate (Y1) at (-3,-2.66);
\coordinate (Y2) at (-0.66,-1.66);
\coordinate  (Y3) at (2,-1.66);
\coordinate  (Y4) at (4,-2);
\coordinate  (Y5) at (2.66,-3);
\coordinate   (Y6) at (-1.35,-3.33);
\draw (Y1) -- (Y2) -- (Y3) -- (Y4) --  (Y5) -- (Y6) -- (Y1);
 \foreach \point in {Y1,Y2,Y3,Y4,Y5,Y6}
    \fill [black,opacity=.5] (\point) circle (2pt);
    
\coordinate (A1) at (-1,-2.5);
\coordinate (A2) at (0.33,-1.9);
\coordinate  (A3) at (1.33,-2.35);
\coordinate  (X1b) at (0.66,-2.85);
\coordinate  (X2b) at (-2,-2.45);
\draw  [thick,orange](A1) -- (A2) -- (A3)--(A1);
\draw[thick] (Y3)--(A2)-- (Y2) -- (X2b) --(A2);
\draw[thick] (Y1)--(X2b)-- (A1) -- (Y6) --(X2b);
\draw[thick] (A1)--(X1b)-- (Y6) ;
\draw[thick] (Y5)--(X1b)-- (A3) -- (Y5);
\draw[thick] (Y4)--(A3) -- (Y3);

\draw[very thick,blue,->] (-0.3,-2.4) -- (-0.3,-1.4);
\draw[very thick,blue,->] (0.45,-2.2) -- (0.45,-0.8);
\draw[very thick,blue,->] (0.9,-2.3) -- (0.9,-1.2);

 \end{tikzpicture}
\caption{The lifting on the polygon.}\label{a}
\end{figure}


We construct  a pyramid $\Pi$ in $\R^3$ lying above $K$ by fixing the boundary $K_0$ 
as the base of the pyramid in the horizontal plane $Q= \R^2$ in $\R^3$.   
For $i=0,\ldots, n+1$, we consider the  planes $P_i$ parallel to $Q$ 
having distance $n-i+1$ relative to the base plane $Q$ in $\R^3$. 
Now, for $i=1,\ldots, n$, we denote by $x_i$ the {orthogonal} projection of the point $y_i$ 
to the plane $P_i$.  
In the plane $P_{0}$, we denote by  $u_1, u_2, u_3$ the {orthogonal} 
projections of the points   $a_1, a_2, a_3$
(see figure \ref{Pyramide5} (a)).

\begin{figure}[H] 
\begin{tikzpicture}  [scale= 0.55]

\coordinate [label=left:$u_1$]  (A1) at (-1,8);
 \coordinate [label=above:$u_2$]  (A2) at (0.33,8.66);
 \coordinate [label=right:$u_3$]  (A3) at (1.33,8);
 \foreach \point in {A1,A2,A3}
    \fill [black,opacity=.5] (\point) circle (2pt);
\draw[thick,orange]  (A1) -- (A2) -- (A3)-- (A1);

 \coordinate [label=left:$b_1$]  (B1) at (-3,-1);
 \coordinate [label=below:$b_2$]  (B2) at (-0.76,0);
 \coordinate [label=below:$b_3$]  (B3) at (1.8,0.2);
  \coordinate [label=right:$b_4$]  (B4) at (4.5,-0.33);
 \coordinate [label=below:$b_5$]  (B5) at (2.66,-1.33);
 \coordinate [label= left:$b_6$]  (B6) at (-1.35,-1.66);

 \foreach \point in {B1,B2,B3,B4,B5,B6}
    \fill [black,opacity=.5] (\point) circle (2pt);
\draw [thick, dashed] (B1)  -- (B2) -- (B3) -- (B4);
\draw  (B6) -- (B5) -- (B4);

 \coordinate [label=left:$x_2$]  (X2) at (0.66,3);
 \coordinate [label=left:$x_1$]  (X1) at (-2,5.72);
 \foreach \point in {X1,X2}
    \fill [black,opacity=.5] (\point) circle (2pt);

\draw (X1) -- (B1) -- (B6) -- (X1)-- (A1) -- (B6) -- (X2) -- (A1);
\draw (B4) -- (A3)-- (B5) -- (X2) -- (A3);
\draw [dashed] (A2) -- (B2) -- (X1) -- (A2) -- (B3) -- (A3);

\draw (-3.76,7.66) node {$P_{0}$};
\draw (-4,8) --(-4.66, 7.33) -- (4.33,7.33) -- (5,8);

\draw (-3.76,5.33) node {$P_1$};
\draw (-4,5.66) --(-4.66, 5) -- (4.33,5) -- (5,5.66);

\draw (-3.76,3) node {$P_2$};
\draw (-4,3.33) --(-4.66, 2.66) -- (4.33,2.66) -- (5,3.33);

\draw(-4,1) node{$\cdots$};
\draw(5,1) node{$\cdots$};

\draw (-3.76,-1.66) node {$P_{n+1}$};
\draw (-4,-1.33) --(-4.66, -2) -- (4.33,-2) -- (5,-1.33);

\draw (-3.76,-4.7) node {$Q$};
\draw (-4,-4.5) --(-4.66, -5.16) -- (4.33,-5.16) -- (5,-4.5);
\coordinate (Y1) at (-3,-4.16);
\coordinate (Y2) at (-0.66,-3.16);
\coordinate  (Y3) at (2,-3.16);
\coordinate  (Y4) at (4.4,-3.5);
\coordinate  (Y5) at (2.66,-4.5);
\coordinate   (Y6) at (-1.35,-4.83);
\draw [thick] (Y1) -- (Y2) -- (Y3) -- (Y4) --  (Y5) -- (Y6) -- (Y1);
 \foreach \point in {Y1,Y2,Y3,Y4,Y5,Y6}
    \fill [black,opacity=.5] (\point) circle (2pt);
    
\coordinate (A1) at (-1,-4);
\coordinate (A2) at (0.33,-3.5);
\coordinate  (A3) at (1.53,-4);
\coordinate  (X1b) at (0.66,-4.35);
\coordinate  (X2b) at (-2,-3.95);
\draw  [thick,orange](A1) -- (A2) -- (A3)--(A1);
\draw[thick] (Y3)--(A2)-- (Y2) -- (X2b) --(A2);
\draw[thick] (Y1)--(X2b)-- (A1) -- (Y6) --(X2b);
\draw[thick] (A1)--(X1b)-- (Y6) ;
\draw[thick] (Y5)--(X1b)-- (A3) -- (Y5);
\draw[thick] (Y4)--(A3) -- (Y3);

\draw[very thick,blue,->] (-0.3,-3.9) -- (-0.3,-2.9);
\draw[very thick,blue,->] (0.45,-3.7) -- (0.45,-2.3);
\draw[very thick,blue,->] (0.9,-3.8) -- (0.9,-2.7);

\node at (4.5,-6) {$(a)$};

 \end{tikzpicture}
\quad\quad\quad\quad
\begin{tikzpicture} [scale= 0.65]

 \coordinate [label=left:$u_1$]  (A1) at (-1,8);
 \coordinate [label=above:$u_2$]  (A2) at (0.33,8.66);
 \coordinate [label=right:$u_3$]  (A3) at (1.33,8);
 \foreach \point in {A1,A2,A3}
    \fill [black,opacity=.5] (\point) circle (2pt);
\draw[thick,orange]  (A1) -- (A2) -- (A3)-- (A1);

 \coordinate [label=left:$b_1$]  (B1) at (-3,-1);
 \coordinate [label=below:$b_2$]  (B2) at (-0.76,0);
 \coordinate [label=below:$b_3$]  (B3) at (1.8,0.2);
  \coordinate [label=right:$b_4$]  (B4) at (4.5,-0.33);
 \coordinate [label=below:$b_5$]  (B5) at (2.66,-1.33);
 \coordinate [label= left:$b_6$]  (B6) at (-1.35,-1.66);

 \foreach \point in {B1,B2,B3,B4,B5,B6}
    \fill [black,opacity=.5] (\point) circle (2pt);
\draw [thick, dashed] (B1)  -- (B2) -- (B3) -- (B4);
\draw  (B6) -- (B5) -- (B4);

 \coordinate [label=left:$x_2$]  (X2) at (0.66,3);
 \coordinate [label=left:$x_1$]  (X1) at (-2,5.72);
 \foreach \point in {X1,X2}
    \fill [black,opacity=.5] (\point) circle (2pt);

\draw (X1) -- (B1) -- (B6) -- (X1)-- (A1) -- (B6) -- (X2) -- (A1);
\draw (B4) -- (A3)-- (B5) -- (X2) -- (A3);
\draw [dashed] (A2) -- (B2) -- (X1) -- (A2) -- (B3) -- (A3);

\draw (-3.76,7.66) node {$P_{0}$};
\draw (-4,8) --(-4.66, 7.33) -- (4.33,7.33) -- (5,8);

\draw (-3.76,5.33) node {$P_1$};
\draw (-4,5.66) --(-4.66, 5) -- (4.33,5) -- (5,5.66);

\draw (-3.76,3) node {$P_2$};
\draw (-4,3.33) --(-4.66, 2.66) -- (4.33,2.66) -- (5,3.33);

\draw(-4,1) node{$\cdots$};
\draw(5,1) node{$\cdots$};

\draw (-3.76,-1.66) node {$P_{n+1}$};
\draw (-4,-1.33) --(-4.66, -2) -- (4.33,-2) -- (5,-1.33);

\coordinate(J1) at (8,3.16);
\coordinate (C1) at (intersection of A1--B6 and X1--J1);
\draw [thick,red](X1) -- (C1);

\coordinate  (J2) at (8,4.4);
\coordinate(C2) at (intersection of A1--X2 and C1--J2);
\draw [thick,red](C1) -- (C2);

\coordinate (J3) at (4,5.35);
\coordinate (C3) at (intersection of A3--X2 and C2--J3);
\draw [thick,red](C2) -- (C3);

\coordinate (J4) at (4,5.7);
\coordinate (C4) at (intersection of A3--B5 and C3--J4);
\draw [thick,red](C3) -- (C4);

\coordinate (J5) at (4,7);
\coordinate (C5) at (intersection of A3--B4 and C4--J5);
\draw [thick,red](C4) -- (C5);

\coordinate (J6) at (-2,8.4);
\coordinate (C6) at (intersection of A3--B3 and C5--J6);
\draw [thick,dashed,red](C5) -- (C6);

\coordinate (J7) at (-4,8);
\coordinate (C7) at (intersection of A2--B3 and C6--J7);
\draw [thick,dashed,red](C6) -- (C7);

\coordinate (J8) at (-8,6.5);
\coordinate (C8) at (intersection of A2--B2 and C7--J8);
\draw [thick,dashed,red](C7) -- (C8);

\draw [thick,dashed,red](C8) -- (X1);

 \foreach \point in {X1,C1,C2,C3,C4,C5,C6,C7,C8}
    \fill [red,opacity=.5] (\point) circle (2.3pt);
    

\coordinate (K1) at (8,4);
\coordinate (D1) at (intersection of A3--B5 and X2--K1);
\draw [thick,blue](X2) -- (D1);

\coordinate (K2) at (8,8.5);
\coordinate  (D2) at (intersection of A3--B4 and D1--K2);
\draw [thick,blue](D1) -- (D2);

\coordinate (K3) at (-8,9);
\coordinate (D3) at (intersection of A3--B3 and D2--K3);
\draw [thick,dashed,blue](D2) -- (D3);

\coordinate (K4) at (-8,7);
\coordinate (D4) at (intersection of A2--B3 and D3--K4);
\draw [thick,dashed,blue](D3) -- (D4);

\coordinate (K5) at (-8,5);
\coordinate (D5) at (intersection of A2--B2 and D4--K5);
\draw [thick,dashed,blue](D4) -- (D5);

\coordinate (K6) at (-8,2.5);
\coordinate (D6) at (intersection of X1--B2 and D5--K6);
\draw [thick,dashed,blue](D5) -- (D6);

\coordinate (K7) at (-8,1.8);
\coordinate (D7) at (intersection of X1--B1 and D6--K7);
\draw [thick,dashed,blue](D6) -- (D7);

\coordinate (K8) at (8,-10.5);
\coordinate (D8) at (intersection of X1--B6  and D7--K8);
\draw [thick,blue](D7) -- (D8);

\coordinate (K9) at (8,-1.5);
\coordinate (D9) at (intersection of A1--B6 and D8--K9);
\draw [thick,blue](D8) -- (D9);

\draw [thick,blue](D9) -- (X2);

 \foreach \point in {X2,D1,D2,D3,D4,D5,D6,D7,D8,D9}
    \fill [blue,opacity=.5] (\point) circle (2.3pt);
    
\node at (3.5,-4) {$(b)$};
    
 \end{tikzpicture}
 \caption{The pyramid $\Pi$ and the decomposition $L'$ of the pyramid $\Pi$.}\label{Pyramide5}
\end{figure}
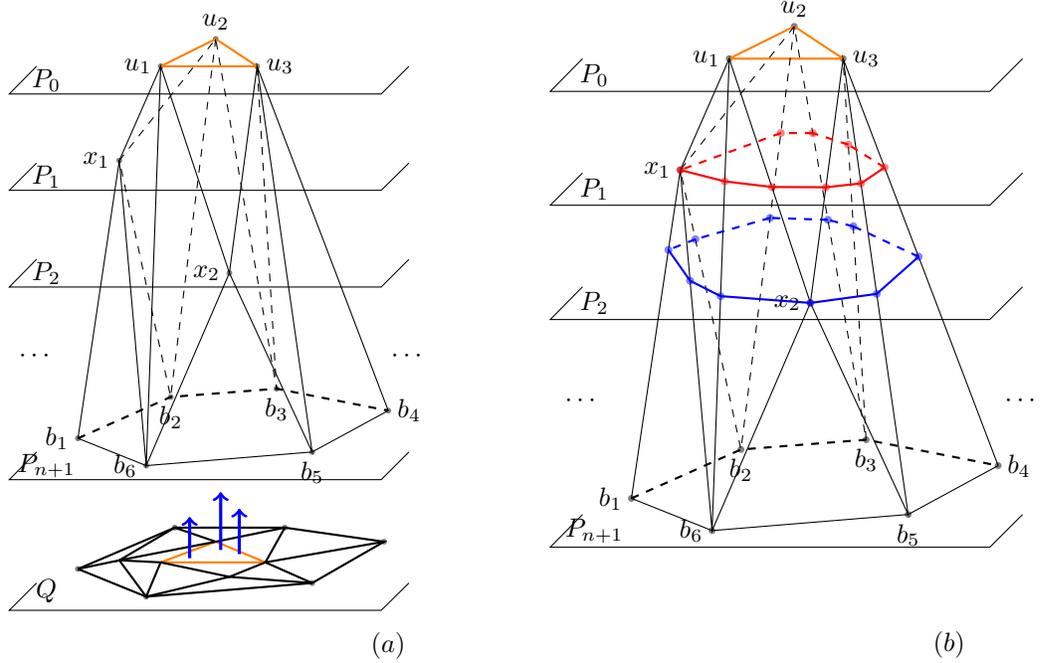

The triangulation of the polygon induces a triangulation $L$ on the pyramid $\Pi$ 
lifting each simplex $[b_i, y_j]$ to $[b_i, x_j]$, 
each $[y_i, y_j]$ to $[x_i, x_j]$ and each $[y_i, a_j]$ to 
$[x_i, u_j]$. 
In the same way, we also lift the 2-dimensional simplices.

\bigskip 

2) Step 2: Let us construct a sub-decomposition 
 $L'$ of the triangulation $L$, 
such that the  intersections of the planes with the pyramid are triangulated in the following way: 
Let us define new vertices of $L'$ as the intersection of $1$-dimensional simplices of $L$ 
with the planes  $P_i$. 
In the same way, we define also new $1$-dimensional simplices of $L'$ 
 as the intersections of  $2$-dimensional simplices of $L$  with  the planes $P_i$. 
The decomposition 
$L'$ of the pyramid contains vertices, edges ($1$-dimensional simplices),  
and faces which can be triangles or quadrilaterals. 
 The sum 
$n_0 -n_1 + n_2$ is the same for the triangulation $L$ and the 
decomposition  $L'$ (see figure \ref{Pyramide5} (b)).

\medskip

3) Step 3: Let us define a sub-triangulation $L''$ of $L$ in the following way: each  
quadrilateral is divided into two triangles.  
  The sum 
$n_0 -n_1 + n_2$ is the same for the triangulations $L$ and $L''$ (see figure \ref{Pyramide7}).

\begin{figure}[H] 
\begin{tikzpicture} [scale= 0.8]

 \coordinate [label=left:$u_1$]  (A1) at (-1,8);
 \coordinate [label=above:$u_2$]  (A2) at (0.33,8.66);
 \coordinate [label=right:$u_3$]  (A3) at (1.33,8);
 \foreach \point in {A1,A2,A3}
    \fill [black,opacity=.5] (\point) circle (2pt);
\draw[thick,orange]  (A1) -- (A2) -- (A3)-- (A1);

 \coordinate [label=left:$b_1$]  (B1) at (-3,-1);
 \coordinate [label=below:$b_2$]  (B2) at (-0.76,0);
 \coordinate [label=below:$b_3$]  (B3) at (1.8,0.2);
  \coordinate [label=right:$b_4$]  (B4) at (4.5,-0.33);
 \coordinate [label=below:$b_5$]  (B5) at (2.66,-1.33);
 \coordinate [label= left:$b_6$]  (B6) at (-1.35,-1.66);

 \foreach \point in {B1,B2,B3,B4,B5,B6}
    \fill [black,opacity=.5] (\point) circle (2pt);
\draw [thick, dashed] (B1)  -- (B2) -- (B3) -- (B4);
\draw  (B6) -- (B5) -- (B4);

 \coordinate [label=left:$x_2$]  (X2) at (0.66,3);
 \coordinate [label=left:$x_1$]  (X1) at (-2,5.72);
 \foreach \point in {X1,X2}
    \fill [black,opacity=.5] (\point) circle (2pt);

\draw (X1) -- (B1) -- (B6) -- (X1)-- (A1) -- (B6) -- (X2) -- (A1);
\draw (B4) -- (A3)-- (B5) -- (X2) -- (A3);
\draw [dashed] (A2) -- (B2) -- (X1) -- (A2) -- (B3) -- (A3); 

\draw (-3.76,7.66) node {$P_{0}$};
\draw (-4,8) --(-4.66, 7.33) -- (4.33,7.33) -- (5,8);

\draw (-3.76,5.33) node {$P_1$};
\draw (-4,5.66) --(-4.66, 5) -- (4.33,5) -- (5,5.66);

\draw (-3.76,3) node {$P_2$};
\draw (-4,3.33) --(-4.66, 2.66) -- (4.33,2.66) -- (5,3.33);

\draw(-4,1) node{$\cdots$};
\draw(5,1) node{$\cdots$};

\draw (-3.76,-1.66) node {$P_{n+1}$};
\draw (-4,-1.33) --(-4.66, -2) -- (4.33,-2) -- (5,-1.33);


\coordinate(J1) at (8,3.16);
\coordinate (C1) at (intersection of A1--B6 and X1--J1);
\draw [thick,red](X1) -- (C1);

\coordinate  (J2) at (8,4.4);
\coordinate(C2) at (intersection of A1--X2 and C1--J2);
\draw [thick,red](C1) -- (C2);

\coordinate (J3) at (4,5.35);
\coordinate (C3) at (intersection of A3--X2 and C2--J3);
\draw [thick,red](C2) -- (C3);

\coordinate (J4) at (4,5.7);
\coordinate (C4) at (intersection of A3--B5 and C3--J4);
\draw [thick,red](C3) -- (C4);

\coordinate (J5) at (4,7);
\coordinate (C5) at (intersection of A3--B4 and C4--J5);
\draw [thick,red](C4) -- (C5);

\coordinate (J6) at (-2,8.4);
\coordinate (C6) at (intersection of A3--B3 and C5--J6);
\draw [thick,dashed,red](C5) -- (C6);

\coordinate (J7) at (-4,8);
\coordinate (C7) at (intersection of A2--B3 and C6--J7);
\draw [thick,dashed,red](C6) -- (C7);

\coordinate (J8) at (-8,6.5);
\coordinate (C8) at (intersection of A2--B2 and C7--J8);
\draw [thick,dashed,red](C7) -- (C8);

\draw [thick,dashed,red](C8) -- (X1);


\coordinate (K1) at (8,4);
\coordinate (D1) at (intersection of A3--B5 and X2--K1);
\draw [thick,blue](X2) -- (D1);

\coordinate (K2) at (8,8.5);
\coordinate  (D2) at (intersection of A3--B4 and D1--K2);
\draw [thick,blue](D1) -- (D2);

\coordinate (K3) at (-8,9);
\coordinate (D3) at (intersection of A3--B3 and D2--K3);
\draw [thick,dashed,blue](D2) -- (D3);

\coordinate (K4) at (-8,7);
\coordinate (D4) at (intersection of A2--B3 and D3--K4);
\draw [thick,dashed,blue](D3) -- (D4);

\coordinate (K5) at (-8,5);
\coordinate (D5) at (intersection of A2--B2 and D4--K5);
\draw [thick,dashed,blue](D4) -- (D5);

\coordinate (K6) at (-8,2.5);
\coordinate (D6) at (intersection of X1--B2 and D5--K6);
\draw [thick,dashed,blue](D5) -- (D6);

\coordinate (K7) at (-8,1.8);
\coordinate (D7) at (intersection of X1--B1 and D6--K7);
\draw [thick,dashed,blue](D6) -- (D7);

\coordinate (K8) at (8,-10.5);
\coordinate (D8) at (intersection of X1--B6  and D7--K8);
\draw [thick,blue](D7) -- (D8);

\coordinate (K9) at (8,-1.5);
\coordinate (D9) at (intersection of A1--B6 and D8--K9);
\draw [thick,blue](D8) -- (D9);

\draw [thick,blue](D9) -- (X2);


\draw [thick,green](A1) -- (C3);
\draw [thick,green](X1) -- (D9);
\draw [thick,green](C1) -- (X2);
\draw [thick,green](C4) -- (X2);
\draw [thick,green](C5) -- (D1);
\draw [thick,green](D7) -- (B6);
\draw [thick,green](D2) -- (B5);

\draw [thick,dashed,green](A3) -- (C7);
\draw [thick,dashed,green](C6) -- (D2);
\draw [thick,dashed,green](C7) -- (D3);
\draw [thick,dashed,green](C8) -- (D4);
\draw [thick,dashed,green](C8) -- (D6);
\draw [thick,dashed,green](D3) -- (B4);
\draw [thick,dashed,green](D4) -- (B2);
\draw [thick,dashed,green](D6) -- (B1);

 \end{tikzpicture}
 \caption{The sub-triangulation $L''$ of the pyramid $\Pi$.}\label{Pyramide7}
\end{figure}
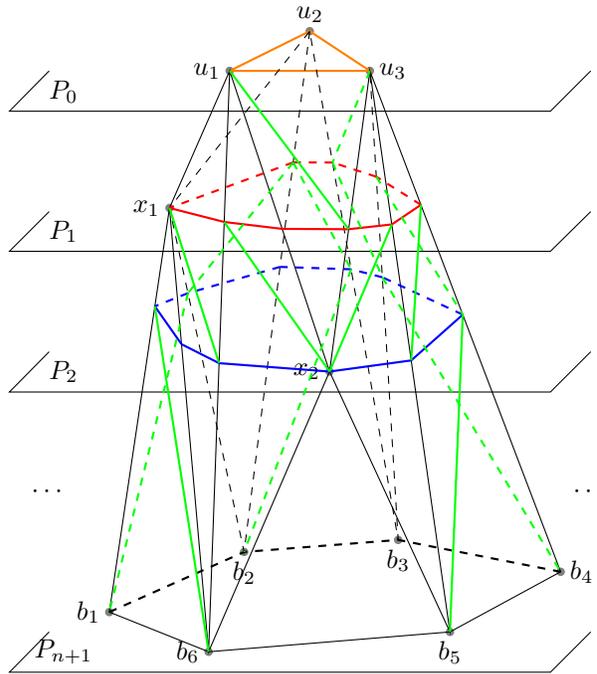

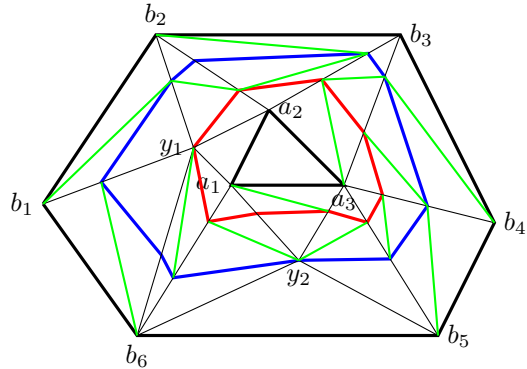
\begin{figure}[H] 
\begin{tikzpicture}  [scale= 0.5]

\coordinate [label=left:$b_1$] (B1) at (-6,-0.5);
\coordinate [label=below:$b_6$] (B2) at (-3.5,-4);
\coordinate [label=right:$b_5$] (B3) at (4.5,-4);
\coordinate [label=right:$b_4$] (B4) at (6,-1);
\coordinate [label=right:$b_3$] (B5) at (3.5,4);
\coordinate [label=above:$b_2$] (B6) at (-3,4);
\draw [very thick] (B1)  -- (B2) -- (B3) -- (B4)--(B5) -- (B6) -- (B1);

\coordinate [label=right:$a_2$]  (A1) at (0,2);
\coordinate [label=left:$a_1$] (A2) at (-1,0);
\coordinate [label=below:$a_3$] (A3) at (2,0);
\draw [very thick] (A1)  -- (A2) -- (A3)-- (A1) ;

\coordinate [label=below:$y_2$] (X1) at (0.8,-2);
\coordinate  [label=left:$y_1$] (X2) at (-2,1);
\coordinate (C11) at (-1.5,-1.5);
\coordinate (C1) at (intersection of X2--C11 and A2--B2);
\coordinate (C21) at (0,-0.7);
\coordinate (C2) at (intersection of C1--C21 and A2--X1);
\coordinate  (C31) at (2,-0.68);
\coordinate  (C3) at (intersection of C2--C31 and A3--X1);
\coordinate (C41) at (3,-1.1);
\coordinate (C4) at (intersection of C3-- C41 and A3--B3);
\coordinate (C51) at (3,-0.3);
\coordinate (C61) at (2.5,1.5);
\coordinate (C71) at (1.5,2.7);
\coordinate (C81) at (-1,2.5);
\coordinate (C5) at (intersection of C4-- C51 and A3--B4);
\coordinate (C6) at (intersection of C5-- C61 and A3--B5);
\coordinate (C7) at (intersection of C6-- C71 and A1--B5);
\coordinate (C8) at (intersection of C7-- C81 and A1--B6);

\draw [very thick,red] (X2) -- (C1)  -- (C2) -- (C3) -- (C4) -- (C5) -- (C6) -- (C7) -- (C8) -- (X2);

\coordinate (X1) at (0.8,-2);
\coordinate (D11) at (3.2,-2);
\coordinate (D21) at (4.2,-0.5);
\coordinate (D31) at (3,3);
\coordinate (D41) at (2.5,3.5);
\coordinate (D61) at (-2.1,3.2);
\coordinate (D6) at (-3.2,2.5);
\coordinate (D51) at (-2.3,3.2);
\coordinate (D71) at (-4.55,0.2);
\coordinate (D81) at (-3.2,-1.2);
\coordinate (D91) at (-2.8,-2.5);
\coordinate (D9) at (intersection of X1--D91 and A2--B2);
\coordinate (D8) at (intersection of D9--D81 and X2--B2);
\coordinate (D7) at (intersection of D8--D71 and X2--B1);
\coordinate (D6) at (intersection of D7--D51 and X2--B6);
\coordinate (D5) at (intersection of D6--D61 and A1--B6);
\coordinate (D4) at (intersection of D5--D41 and A1--B5);
\coordinate (D3) at (intersection of D4--D31 and A3--B5);
\coordinate (D2) at (intersection of D3--D21 and A3--B4);
\coordinate (D1) at (intersection of D2--D11 and A3--B3);
\draw [very thick,blue] (X1) -- (D1)  -- (D2) -- (D3) -- (D4)--(D5) -- (D6) -- (D7) -- (D8) -- (D9)--(X1);

\draw (B1) -- (X2);
\draw (A2) -- (X2);
\draw (A1) -- (X2);
\draw (B6) -- (X2);
\draw (B2) -- (X2);
\draw (B6) -- (A1);

\draw (A2) -- (X1);
\draw (A3) -- (X1);
\draw (B2) -- (X1);
\draw (B3) -- (X1);
\draw (B3) -- (A3);
\draw (B4) -- (A3);
\draw (B5) -- (A3);
\draw (B5) -- (A1);
\draw (B2) -- (A2);

\draw [thick,green](A2) -- (C3);
\draw [thick,green](A3) -- (C7);
\draw [thick,green](X2) -- (D9);
\draw [thick,green](C1) -- (X1);
\draw [thick,green](C4) -- (X1);
\draw [thick,green](C5) -- (D1);
\draw [thick,green](C6) -- (D2);
\draw [thick,green](C7) -- (D3);
\draw [thick,green](C8) -- (D4);
\draw [thick,green](C8) -- (D6);
\draw [thick,green](D7) -- (B2);
\draw [thick,green](D2) -- (B3);
\draw [thick,green](D3) -- (B4);
\draw [thick,green](D4) -- (B6);
\draw [thick,green](D6) -- (B1);

 \end{tikzpicture}

\caption{The sub-triangulation $K''$ of the polygon.} \label{Pyramide71}
\end{figure}

4) Step 4: We will prove that the intersection of $L''$ with each plane $P_i$ 
is a curve homeomorphic to a circle. 
 First, we show that the projection of  $L''$ to the plane $Q$ provides a sub-triangulation $K''$ 
of the polygon $K$ (figure \ref{Pyramide71}).  
In fact, by construction, since there is no vertical edges in the pyramid, then  the {orthogonal} projection $\pi$ of the pyramid to $Q$ is a bijection between the   triangulations $L''$ and $K''$. Notice that each vertex of $K''$ corresponds to a vertex of $L''$ and, in the same way, for the edges 
and the triangles of $K''$ and $L''$, respectively.  
Moreover, in the subdivision $L''$, 
each edge is the common edge of exactly two triangles, and 
likewise in $K''$. 
 That implies that $K''$ is a triangulation. 

Now, we prove, by induction, that the intersection of $L''$ with each plane $P_i$ 
is a curve homeomorphic to a circle. We see that $L \cap P_{0}$, which is the boundary of the triangle  $\sigma_0$ denoted by 
$B_{0}$, is homeomorphic to a circle. Assume that $B_i$ is homeomorphic to a circle but  $B_{i+1}$ is not homeomorphic to a circle. Then $L''$ has multiple points in the plane $P_{i+1}$ (see figure \ref{duplo}).  By projection, $K''$ is no longer a triangulation. 
That provides a contradiction.

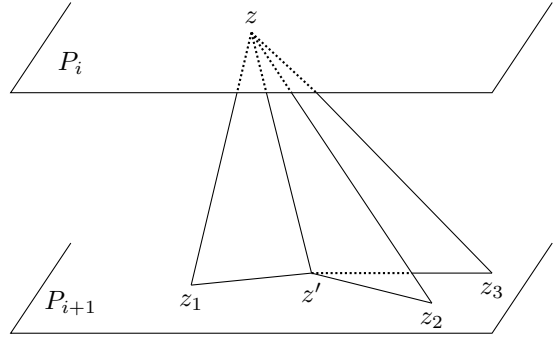
\begin{figure}[H] 
\begin{tikzpicture} [scale= 0.8]

\draw (1, 1.5) -- (0,0) -- (8,0) -- (9, 1.5);
\node at (1,0.5) {$P_{i+1}$};
\draw (1, 5.5) -- (0,4) -- (8,4) -- (9, 5.5);
\node at (1,4.5) {$P_{i}$};


\draw[thick,densely dotted] (4, 5) -- (3.7619,4); 
\draw (3.7619,4) -- (3, 0.8); 
\node at (4,5) [above]{$z$};
\node at (3, 0.8) [below]{$z_1$};


\draw[thick,densely dotted] (4, 5) -- (4.25,4); 
\draw (4.25,4) -- (5, 1); 
\node at (5,1) [below]{$z'$};


\draw[thick,densely dotted] (4, 5) -- (4.6666,4); 
\draw (4.6666,4) -- (7, 0.5); 
\node at (7,0.5) [below]{$z_2$};

\draw[thick,densely dotted] (4, 5) -- (5.0811,4); 
\draw (5.0811,4) -- (8, 1); 
\node at (8,1) [below]{$z_3$};

\draw (3, 0.8) -- (5,1); 

\draw (5,1) -- (7,0.5); 

\draw[thick,densely dotted] (5,1) -- (6.6667, 1); 
\draw (6.6667, 1) -- (8,1); 

\end{tikzpicture}
\medskip
\caption{Not admisible picture: The intersection $B_i$ of $L$ with each plane  $P_i$ can not have multiple points.}\label{duplo}
\end{figure}

5) Step 5: Apply Cauchy's method (section \ref{methode_Cauchy}): 
Now, let us apply Cauchy's method on the pyramid starting 
 by removing 
the triangle $\sigma_0$. 
 Assume that we have already removed all triangles above the plane $P_i$. 
We will prove that if we remove all the triangles in the band  situated between $P_i$ and $P_{i+1}$, the sum 
 $n_0 - n_1 + n_2$ does not change. 
This fact can be established 
processed
 since the open band between 
 $B_{i}$ and $B_{i+1}$ does not possess vertices, as follows: 
 Let us fix a triangle $(\alpha_0, \alpha_1,\beta_0)$  of the band between $B_{i}$ and $B_{i+1}$, 
where the vertices $\alpha_0$ and $\alpha_1$ belong to $B_i$ and $\beta_0$ belongs to  $B_{i+1}$. 
First, {we remove} the triangle $(\alpha_0, \alpha_1,\beta_0)$ by operation I, 
without changing the sum $n_0 - n_1 + n_2$. 
Now, the edge  $(\alpha_1,\beta_0)$ is an edge of either the triangle $(\alpha_1,\beta_0, \beta_1)$, where $\beta_1 \in B_{i+1}$ (see figure \ref{situa} (a)), 
or of the triangle $(\alpha_1, \alpha_2,\beta_0)$, where $\alpha_2 \in B_{i}$ (see figure \ref{situa} (b)). 
In the first case, the triangle $(\alpha_1,\beta_0, \beta_1)$ can be removed by operation I of  
Cauchy's process, and in the second case 
 the triangle $(\alpha_1, \alpha_2,\beta_0)$ can be removed by operation II.  In both of these two cases, the sum $n_0 - n_1 + n_2$ is not changed.

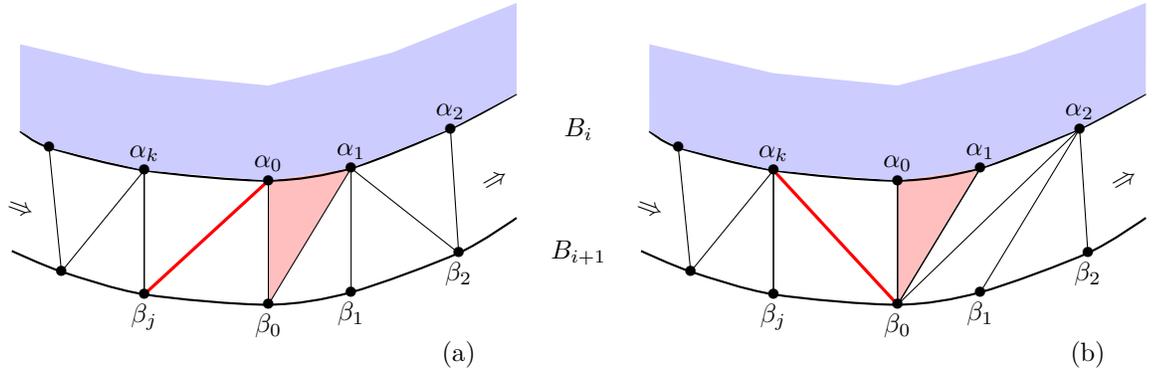
\begin{figure} [H]

 \begin{tikzpicture}  [scale= 0.55]
 
 \node (1) at (-6,2.5) [rotate=340] {$\Rightarrow$};
\coordinate (a) at (-5.3,4) ;
\coordinate (ak) at (-3,03.45);
\coordinate (a0) at (0,3.2) ;
\coordinate (a1) at (2,3.5) ;
\coordinate (a2) at (4.4,4.45) ;

\node (Bi) at (7.5,4.45) {$B_i$};
          
\coordinate (b) at (-5,1);
\coordinate (bj) at (-3,0.45); 
\coordinate (b0) at (0,0.2) ;
\coordinate (b1) at (2,0.5) ;
\coordinate (b2) at (4.6,1.45) ;     

\node (Bi1) at (7.5,1.45) {$B_{i+1}$};

\node (2) at (5.5,3.2)[rotate=30]  {$\Rightarrow$};

 \fill[fill=pink]   (a0) -- (b0)-- (a1) -- (a0);
 
\draw[thick] plot[smooth] coordinates {(-6,4.4)(-5.3,4) (-3,03.45) (0,3.2) (2,3.5) (4.4,4.45) (6,5.3)};     
\draw[thick] plot[smooth] coordinates {(-6.2,1.5)(-5,1) (-3,0.45) (0,0.2) (2,0.5) (4.6,1.45) (6,2.3)};     

\coordinate (x1) at (-6,6.5);
\coordinate (x2) at (-3,5.8);
\coordinate (y) at (0,5.5);
\coordinate (x3) at (3,6.3);
\coordinate (x4) at (6,7.5);
 \fill[fill=blue!20]  (x1) -- (-6,4.4)-- (-5.3,4)-- (-3,03.45) --(0,3.2) -- (2,3.5)-- (4.4,4.45) --(6,5.3)-- 
 (x4) -- (x3) -- (y) -- (x2) -- (x1) ;

\draw (a) -- (b) --(ak) --(bj) ;
 \draw[very thick,red] (a0) -- (bj);
\draw (a0) -- (b0) -- (a1) -- (b2) -- (a2);
\draw (ak) -- (bj);
\draw (a1) -- (b1);

\filldraw (a) node {$\bullet$} ;
\filldraw (ak) node {$\bullet$} node [above] {$\alpha_k$};
\filldraw (a0) node {$\bullet$} node [above] {$\alpha_0$};
\filldraw (a1) node {$\bullet$} node [above] {$\alpha_1$};
\filldraw (a2) node {$\bullet$} node [above] {$\alpha_2$};

\filldraw (b) node {$\bullet$} ;
\filldraw (bj) node {$\bullet$} node [below] {$\beta_j$};
\filldraw (b0) node {$\bullet$} node [below] {$\beta_0$};
\filldraw (b1) node {$\bullet$} node [below] {$\beta_1$};
\filldraw (b2) node {$\bullet$} node [below] {$\beta_2$};

\node at (4.6,-1) {(a)};
 \end{tikzpicture}
 \begin{tikzpicture} [scale= 0.55]

 \node (1) at (-6,2.5) [rotate=340]{$\Rightarrow$};
\coordinate (a) at (-5.3,4) ;
\coordinate (ak) at (-3,03.45);
\coordinate (a0) at (0,3.2) ;
\coordinate (a1) at (2,3.5) ;
\coordinate (a2) at (4.4,4.45) ;
          
\coordinate (b) at (-5,1);
\coordinate (bj) at (-3,0.45); 
\coordinate (b0) at (0,0.2) ;
\coordinate (b1) at (2,0.5) ;
\coordinate (b2) at (4.6,1.45) ;         

\node (2) at (5.5,3.2)[rotate=30] {$\Rightarrow$};

 \fill[fill=pink]   (a0) -- (b0)-- (a1) -- (a0);
 
\draw[thick] plot[smooth] coordinates {(-6,4.4)(-5.3,4) (-3,03.45) (0,3.2) (2,3.5) (4.4,4.45) (6,5.3)};     
\draw[thick] plot[smooth] coordinates {(-6.2,1.5)(-5,1) (-3,0.45) (0,0.2) (2,0.5) (4.6,1.45) (6,2.3)};     

\coordinate (x1) at (-6,6.5);
\coordinate (x2) at (-3,5.8);
\coordinate (y) at (0,5.5);
\coordinate (x3) at (3,6.3);
\coordinate (x4) at (6,7.5);
 \fill[fill=blue!20]  (x1) -- (-6,4.4)-- (-5.3,4)-- (-3,03.45) --(0,3.2) -- (2,3.5)-- (4.4,4.45) --(6,5.3)-- 
 (x4) -- (x3) -- (y) -- (x2) -- (x1) ;

\draw (a) -- (b) --(ak) --(bj) ;
 \draw[very thick,red] (ak) -- (b0);
\draw (a0) -- (b0) -- (a1);
\draw (a0) -- (b0) -- (a1);
\draw (b0) -- (a2);
\draw (ak) -- (bj);
\draw (b2)--(a2) -- (b1);

\filldraw (a) node {$\bullet$} ; 
\filldraw (ak) node {$\bullet$} node [above] {$\alpha_k$};
\filldraw (a0) node {$\bullet$} node [above] {$\alpha_0$};
\filldraw (a1) node {$\bullet$} node [above] {$\alpha_1$};
\filldraw (a2) node {$\bullet$} node [above] {$\alpha_2$};

\filldraw (b) node {$\bullet$} ;
\filldraw (bj) node {$\bullet$} node [below] {$\beta_j$};
\filldraw (b0) node {$\bullet$} node [below] {$\beta_0$};
\filldraw (b1) node {$\bullet$} node [below] {$\beta_1$};
\filldraw (b2) node {$\bullet$} node [below] {$\beta_2$};
\node at (4.6,-1) {(b)};
  \end{tikzpicture}

\caption{Going from $B_i$ to $B_{i+1}$.}\label{situa}
\end{figure}

We continue the process for all the triangles of the band,  
all of which are one of two cases above, 
until we reach the last vertices of $B_i$ and $B_{i+1}$ situated before going back to  $\alpha_0$ and $\beta_0$, respectively. 
We call these vertices $\alpha_k$ and $\beta_j$. 
We have two possible situations (a) and (b) 
(see figure \ref{situa}). 
In situation (a), the last remaining triangles are $(\alpha_k, \alpha_0, \beta_j)$ 
and $(\alpha_0, \beta_j, \beta_0)$. 
In this case, these triangles  
can be removed in this order using  
operation II.  
In situation (b), the remaining triangles $(\alpha_k, \beta_j,\beta_0)$  and $(\alpha_k, \alpha_0,\beta_0)$
 can be removed in this order using also operation II. In the two cases, the sum $n_0 - n_1 + n_2$ is not changed.

\medskip 

6) Step 6: The conclusion. 

\medskip

The process continues until we reach the boundary 
 $B_{n+1}$ of the hole, which is the boundary $K_0$  of $K$,  
and is also the intersection of $L''$ with the plane $P_{n+1}$.

Let us denote by $n_0^K$, $n_1^K$ and $n_2^K$ respectively the numbers of vertices, edges and triangles of the triangulation $K$ and let us use the same notation for the sub-triangulation $K''$ and the triangulation $K_0$. We have 
$$n_0^K - n_1^K + n_2^K = n_0^{K''} - n_1^{K''}  + n_2^{K''} = n_0^{K_0}  - n_1^{K_0} +1.$$ 
The first equality follows from the fact that the sub-triangulation process does not change 
the alternating sum and the second equality comes 
from the fact that we removed the triangle $\sigma_0$ 
at the beginning and that $n_2^{K_0} = 0$. Here  we take the identifications of simplices of $K_0$ into account.

Given a triangulation $K$ of the polygon, the result does not depend on the choices made. 
\end{proof}

\begin{proof}[{\bf Proof of theorem \ref{Euler's Theorem} using Cauchy's method}]
Let $\widehat{K}$ be a convex polyhedron. 
We proceed with the planar representation $K$ of $\widehat{K}$, according to the 
first step of Cauchy's proof (see figure \ref{rep-Cauchy}). Notice that $K$ is a polygon without any identification of simplices on its boundary $K_0$. 
Then $n_0^{K_0} - n_1^{K_0} + n_2^{K_0} = 0.$ 
Theorem \ref{teo1} implies that $n_0^{\widehat{K}} - n_1^{\widehat{K}} + n_2^{\widehat{K}} = +2$, taking into account the removed polygon ${\mathcal P}$ in the first step of Cauchy's proof. 
Since theorem \ref{teo1}  is proved using only Cauchy's method, then Euler's formula is also proved by using only Cauchy's method.
\end{proof} 

Before going further with applications, let us provide some remarks on the proof of theorem \ref{teo1}.

\begin{remark}
The stereographic projection proof of 
 Euler's formula is a particular case of the proof.
\end{remark}

\begin{remark}
There are other ways to define an order of the vertices of the polygon to be able to draw the pyramid, 
without using the Euclidean distance in $\R^2$. 

One possible way  is the use of notion of distance between  two vertices as
the least number of edges in an edge path joining them. As in the proof of theorem \ref{teo1}, 
choose a $2$-dimensional simplex
$\sigma_0$ in the interior of the polygon $K$ and define distance $0$ for the three vertices of $\sigma_0$.
 Determining any order between vertices whose distance to vertices of  $\sigma_0$ is 1, 
one continues the ordering determining 
any order between 
vertices whose distance to vertices of  $\sigma_0$ is 2, etc. Then one proceeds with 
 the construction of the pyramid. 

Another way would be to start the proof of theorem \ref{Euler's Theorem} 
with the convex polyhedron and order the $2$-dimensional faces
according to the shelling process (see \cite{Zie} and \cite{BM}). The $2$-dimensional faces of
the polyhedron are dual to  the vertices of the polar polyhedron. 
One obtains an 
order on the vertices of the polar polyhedron. One continues 
the proof using the polar polyhedron 
instead of the original polyhedron, knowing that the sum $n_0-n_1+n_2$ is the 
same for the polyhedron and its polar. 
 However, the shelling is a tool which was defined well after 
Cauchy's time, so it is not acceptable in our context. We mention it just for the sake of completeness. 
\end{remark}

\begin{remark}
In step 4 of the  proof, the projection on the plane $Q$ of the intersection of each plane $P_i$ 
with the pyramid  is a Jordan curve passing through $y_i$. 
Moreover, in step 5 of our proof, we make it clear that if the boundary of the extended hole is  
homeomorphic to a circle, then Cauchy's process works.  
It is then possible to proceed in 
step 5 either with the sub-triangulation $L''$ of  
the pyramid or with the sub-triangulation $K''$ of  the 
polygon $K$.
\end{remark}

\begin{remark}
In step 5 of the proof we use only the operations I and II of Cauchy.
Observing that if we change the order of removal 
 of the last remaining triangles, 
for example, in situation (a), if we remove the triangle $(\alpha_0, \beta_j, \beta_0)$ and thereafter the triangle $(\alpha_k, \alpha_0, \beta_j)$,  
we will use first operation I of Cauchy and then the operation that we called  operation III in  section \ref{depois_Cauchy} (see figure \ref{figuraCE} (d)). 
Here also, we do not change the sum $n_0 - n_1 + n_2$. 

\end{remark}

\begin{remark} 
As we emphasized at the beginning of this section, 
 we use in the  proof only Cauchy's method 
(section \ref{methode_Cauchy})  without other tools. 
We know very well that there exist ``modern and faster'' ways to prove theorem \ref{teo1}. 
However, these proofs use tools that appeared after Cauchy's time, 
in particular some proofs use the Jordan Lemma, which, as we have seen,  
appears as an artifact  in our proof.
\end{remark}

\section{ Applications}\label{applications}

In the following, using theorem \ref{teo1} as the main tool, 
we prove that the sum $n_0 - n_1 + n_2$ does not depend on the triangulation in the case of 
the sphere, the torus, the projective plane, the Klein bottle and even for a singular surface: the pinched torus.  In each case, we use a planar {representation}  of the surface homeomorphic to a disc 
with possible identifications on the boundary, and the following lemma 
concerning the ``cutting surfaces''  technique that was introduced 
by Alexander Veblen in a seminar in 1915 
(see \cite{Bra}). This idea is well developped in the book by Hilbert and Cohn-Vossen 
\cite{HC}, in particular for the surfaces that we give as examples.

The following ``cutting'' lemma will be used in the forthcoming proofs.

\begin{lem} \label{lemageneral}
Let $T$ be a triangulation of a compact surface. Let $\Gamma$  be a continuous simple curve in $S$. 
There exists a sub-triangulation $T'$ of $T$, with 
 curvilinear simplices,  
compatible with the curve (that means $\Gamma$ is an union of segments of $T'$) 
 in such the way that the number 
$n_0 - n_1 + n_2$ is the same for $T$ and $T'$. 
\end{lem}

\begin{preuve} 
First of all, we can assume that the curve $\Gamma$ is transversal to all the edges of $T$, {\it i.e} 
 the intersection of $\Gamma$  with each edge is a finite number of points.  
Otherwise, a small perturbation of $\Gamma$ allows us to obtain transversality.

We choose a base point  (the starting point) $x_0$ on the curve, 
as well as an {orientation} of the curve. 
If the curve is not closed, we define the base point as one of the two extremities. 
The following process does not depend on either the starting point, or the orientation of the curve.  

A sub-triangulation $T'$ is built simplex by simplex following the orientation of the curve $\Gamma$. 
The first subdivided simplex  is the one 
 $\sigma_0$ containing the base point. 
 Let $y$ be the first point where  the curve leaves $\sigma_0$. 
The (curvilinear)  segment  $(x_0, y)$ will be an edge of $T'$ as well as  
segments connecting $x_0$ to vertices of $\sigma_0$, 
one of 
 which can be $(x_0, y)$ if the point $y$ is a vertex of $\sigma_0$ (figure \ref{subdivi 0}).

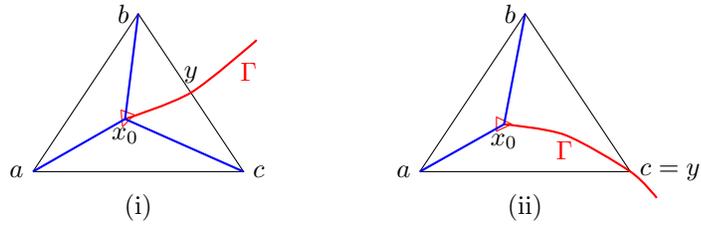
\begin{figure} [H] 
 \begin{tikzpicture}  [scale= 0.35]
  
\draw (-4,0) node[left]{$a$}   --  (0,6) node[left]{$b$} -- (4,0) node[right]{$c$} -- (-4,0);
\draw[thick,red] plot[smooth] coordinates {(-0.5,2) (2,3) (4.5,5)};     
\node at (-0.5,2) [rotate=70, red, very thick]  {$\vartriangleleft$};
\coordinate [label=above:$y$] (y) at (2,3);
\coordinate [label=below:$x_0$] (x_0) at (-0.5,2);
\draw [thick,blue](-4,0) -- (-0.5,2) -- (0,6);
\draw [thick,blue] (-0.5,2) -- (4,0);
\coordinate (u) at (-4,0);
  \draw (4.2,4.5) node[red,below] {$\Gamma$} ;
    \draw (0,-0.5) node[below] {$\rm (i)$} ;
\end{tikzpicture}
\qquad\qquad 
\begin{tikzpicture}  [scale= 0.35] 
    
\draw (-4,0) node[left]{$a$}   --  (0,6) node[left]{$b$} -- (4,0) node[right]{$c=y$} -- (-4,0);
\draw[thick,red] plot[smooth] coordinates {(-0.8,1.8)(1.5,1.4) (4,0) (5,-1)};    
\node at (-0.8,1.8) [rotate=180, red, very thick]  {$\vartriangleleft$}; 
\draw [thick,blue] (-4,0) -- (-0.8,1.8) -- (0,6);
\coordinate (u) at (-4,0);
\coordinate (v) at (4,0);
\coordinate [label=below:$x_0$] (x_0) at (-0.8,1.8);
   \draw (1.5,1.5) node[red,below] {$\Gamma$} ;
       \draw (0,-0.5) node[below] {$\rm (ii)$} ;
\end{tikzpicture}

\caption{Subdivision of the first simplex I.}\label{subdivi 0} 
\end{figure}

Now, it is enough to perform the construction for a simplex $\sigma = (a,b,c)$ 
the curve $\Gamma$  enters. 
In the following construction, we assume that all the simplices 
that the curve meets 
between the base point and the simplex  $\sigma$ (with the given orientation) are already subdivided. 
 The ``entry'' point of $\Gamma$  in the simplex $\sigma$ can be either a 
{vertex} or a point $d$ located on an edge of $\sigma$.

If the entry point of the curve $\Gamma$  is a vertex $a$, 
the curve can exit at a point  $d$ located either in the opposite edge, 
or in an edge containing the vertex $a$, or at another vertex. 
In the first case ({figure} \ref{subdivi I} (i)), 
we divide the triangle $(a,b,c)$ into two (curvilinear) triangles $(a,b,d)$ and $(a,d,c)$. 
In the second case (figure \ref{subdivi I} (ii)),  
let $e\in (a,c)$ be the point  at which  the curve  $\Gamma$  exits from the 
  triangle and we choose a point $f$ on the curve, located between $a$ and $e$. 
We divide the triangle $(a,b,c)$ into four (curvilinear) triangles $(a,b,f)$, $(b,f,e)$, $(b,e,c)$ and  $(a,f,e)$. 
Finally, in the last case (figure \ref{subdivi I} (iii)), 
assume that the exit point is the vertex $c$, 
we choose one point $f$ on the curve, located between $a$ and $c$.  
We divide the triangle $(a,b,c)$ into three (curvilinear) triangles 
 $(a,b,f)$, $(b,f,c)$ and $(a,f,c)$.

\begin{figure} [H] 
 \begin{tikzpicture}  [scale= 0.35]
  
\draw (-4,0) node[left]{$a$}   --  (0,6) node[left]{$b$} -- (4,0) node[right]{$c$} -- (-4,0);
\draw[thick,red] plot[smooth] coordinates {(-5,-1)(-4,0)(0,2) (2,3) (4.5,5)};     
\node at (-4.85,-0.8) [rotate=240, red, very thick]  {$\vartriangleleft$};
\coordinate [label=above:$d$] (d) at (2,3);
\coordinate (u) at (-4,0);
 \foreach \point in {u,d}
    \fill [black,opacity=.5] (\point) circle (3pt);
       \draw (4.2,4.5) node[red,below] {$\Gamma$} ;
    \draw (0,-0.5) node[below] {$\rm (i)$} ;
\end{tikzpicture}
\qquad
\begin{tikzpicture}  [scale= 0.35]
  
\draw (-4,0) node[left]{$a$}   --  (0,6) node[left]{$b$} -- (4,0) node[right]{$c$} -- (-4,0);
\draw[thick,red] plot[smooth] coordinates {(-5,-1)(-4,0)(-1,1.5) (1.5,0) (2.2,-1)};     
\node at (-4.87,-0.8) [rotate=240, red, very thick]  {$\vartriangleleft$};
\draw [thick,blue](-1,1.5) -- (0,6) -- (1.5,0);
\coordinate (u) at (-4,0);
\coordinate [label=below:$e$] (e) at (1.5,0);
\coordinate [label=below:$f$] (f) at (-1,1.5);
 \foreach \point in {u,e,f}
    \fill [black,opacity=.5] (\point) circle (3pt);
   \draw (2.6,0) node[red,below] {$\Gamma$} ;
       \draw (0,-0.5) node[below] {$\rm (ii)$} ;
\end{tikzpicture}
\qquad
\begin{tikzpicture}  [scale= 0.35]
  
\draw (-4,0) node[left]{$a$}   --  (0,6) node[left]{$b$} -- (4,0) node[right]{$c$} -- (-4,0);
\draw[thick,red] plot[smooth] coordinates {(-5,-1)(-4,0)(0.8,1.8) (4,0) (5,-1)};    
\node at (-4.9,-0.8) [rotate=240, red, very thick]  {$\vartriangleleft$}; 
\draw [thick,blue] (0,6) -- (0.8,1.8);
\coordinate (u) at (-4,0);
\coordinate (v) at (4,0);
\coordinate [label=below:$f$] (f) at (0.8,1.8);
 \foreach \point in {u,v,f}
    \fill [black,opacity=.5] (\point) circle (3pt);
\draw (-1.4,1.65) node[red] {$\Gamma$} ;
       \draw (0,-0.5) node[below] {$\rm (iii)$} ;
\end{tikzpicture}

\caption{Subdivision II.}\label{subdivi I} 
\end{figure}
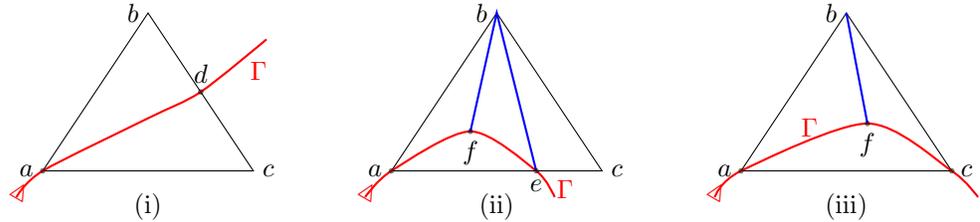

Notice that, in all the three cases, the choice of sub-triangulation is not unique, but the sum
$n_0 - n_1 + n_2$ 
remains unchanged 
independent of the choice.

If the entry point of the curve $\Gamma$ is located in one edge, we denote by $d$  
the entry point and by $(a,c)$ the edge of $\sigma$ containing $d$.  
The next exit point of $\Gamma$  can be 
either in an edge different from $(a,c)$ (for example $(a,b)$), 
or in the same edge $(a,c)$, or it can be a vertex (see figure \ref{subdivi II}).

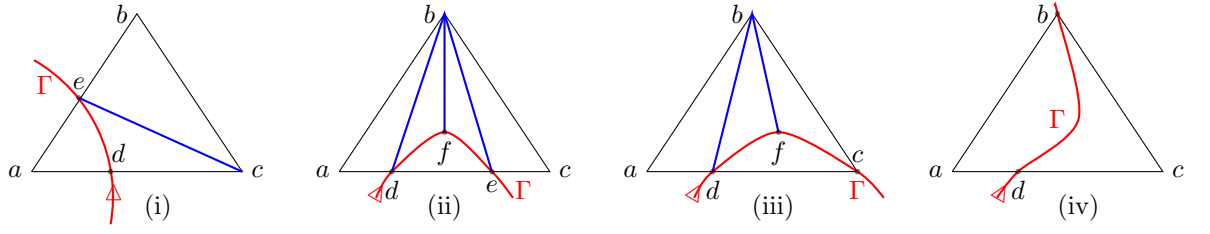
\begin{figure} [H] 

 \begin{tikzpicture}  [scale= 0.35]
  
\draw (-4,0) node[left]{$a$}   --  (0,6) node[left]{$b$} -- (4,0) node[right]{$c$} -- (-4,0);
\draw [thick,red](-1,-2) arc (-10:60: 6 );
\node at (-0.87,-0.8) [rotate=270, red, very thick]  {$\vartriangleleft$};
\coordinate [label=above:$e$] (e) at (-2.2, 2.8);
\coordinate (f) at (-1, 0);
\coordinate [label=above:$d$] (d) at (-0.7, 0);
\draw [thick,blue](e) -- (4,0);
 \foreach \point in {e,f}
    \fill [black,opacity=.5] (\point) circle (3pt);
\draw (-3.5,3.3) node[red] {$\Gamma$} ;
    \draw (0.8,-0.5) node[below] {$\rm (i)$} ;
\end{tikzpicture}
\quad
 \begin{tikzpicture}  [scale= 0.35]
  
  \draw (-4,0) node[left]{$a$}   --  (0,6) node[left]{$b$} -- (4,0) node[right]{$c$} -- (-4,0);
  
\draw[thick,red] plot[smooth] coordinates {(-2.6,-1)(-2,0) (0,1.5) (1.8,0) (2.6,-1)};     
\node at (-2.5,-0.8) [rotate=240, red, thick]  {$\vartriangleleft$}; 
\draw [thick,blue](-2,0) -- (0,6) -- (1.8,0);
\draw [thick,blue](0,6) -- (0,1.5);
\coordinate [label=below:$d$] (d) at (-2, 0);
\coordinate [label=below:$f$] (f) at (0,1.5);
\coordinate [label=below:$e$] (e) at (1.8, 0);
 \foreach \point in {d,e,f}
    \fill [black,opacity=.5] (\point) circle (3pt);
   \draw (3,0) node[red,below] {$\Gamma$} ;
\draw (0,-0.5) node[below] {$\rm (ii)$} ;
\end{tikzpicture}
\quad
\begin{tikzpicture}  [scale= 0.35]
  
\draw (-4,0) node[left]{$a$}   --  (0,6) node[left]{$b$} -- (4,0) node[above]{$c$} -- (-4,0);
\draw[thick,red] plot[smooth] coordinates {(-2.2,-1)(-1.5,0) (1,1.5) (4,0)(5,-1)};
\node at (-2.13,-0.8) [rotate=240, red, thick]  {$\vartriangleleft$}; 
\draw [thick,blue](-1.5,0) -- (0,6) -- (1,1.5);
\coordinate (v) at (4,0);
\coordinate [label=below:$d$] (d) at (-1.5,0);
\coordinate [label=below:$f$] (f) at (1,1.5);
 \foreach \point in {v,d,f}
    \fill [black,opacity=.5] (\point) circle (3pt);
\draw (4,0) node[red,below] {$\Gamma$} ;
       \draw (0.8,-0.5) node[below] {$\rm (iii)$} ;
\end{tikzpicture}
\quad
\begin{tikzpicture}  [scale= 0.35]
  
\draw (-4,0) node[left]{$a$}   --  (0,6) node[left]{$b$} -- (4,0) node[right]{$c$} -- (-4,0);
\draw[thick,red] plot[smooth] coordinates {(-2.3,-1)(-1.5,0) (0.8,2) (0,6)(-0.1,6.3)};
\node at (-2.13,-0.8) [rotate=240, red, thick]  {$\vartriangleleft$}; 
\coordinate (v) at (0,6);
\coordinate [label=below:$d$] (d) at (-1.5,0);
 \foreach \point in {v,d}
    \fill [black,opacity=.5] (\point) circle (3pt);
\draw (0,2) node[red] {$\Gamma$} ;
       \draw (0.8,-0.5) node[below] {$\rm (iv)$} ;
\end{tikzpicture}
\caption{Subdivision III. } \label{subdivi II}
\end{figure}

In the first 
case ({figure} \ref{subdivi II} (i)), 
we define, for example, 
a sub-triangulation of the triangle $(a,b,c)$ formed by the (curvilinear) triangles $(a,d,e)$, $(c,e,d)$ and $(c,e,b)$.

In the second case, 
 for example if the curve $\Gamma$  enters and exits by two points $d$ and $e$ situated on the same segment $(a,c)$, we choose a point $f$ on the curve located between 
 $d$ and $e$ (figure \ref{subdivi II} (ii)). 
We define a  sub-triangulation of the triangle  $(a,b,c)$ 
 formed by five (curvilinear) triangles $(a,b,d)$, $(b,d,f)$, $(b,f,e)$, $(b,e,c)$ and $(d,f,e)$. 

The two last cases 
 of figure \ref{subdivi II} (iii) and (iv))  
are similar to the cases of figure 
\ref{subdivi I} (ii) and (i)  respectively. 
 
In the four cases, 
the choice of sub-triangulations is not unique, 
but the sum $n_0 - n_1 + n_2$ remains unchanged 
 independent of the choice. 

The process continues for all the $2$-dimensional simplices 
crossed by the curve $\Gamma$, 
 and it is completed in a finite number of steps 
because the number of simplices is finite, even after the subdivision. 
\end{preuve} 
\goodbreak

\subsection {The sphere case.}  

In the following, we provide an alternating proof of Euler's formula for the sphere using   theorem \ref{teo1} as the main tool, with the idea of ``cutting surfaces'' as an additional tool.

Let $T$ be a triangulation of the sphere $\Sp^2$. 
We consider four curves on the sphere: 
The equator $E$ (or any parallel) and three curves $\gamma_1$, $\gamma_2$ and $\gamma_3$, 
going from the North pole $N$ to the curve $E$ along meridians. 
Let us denote by 
$a_i $ the intersection points $\gamma_i \cap E$, where $i=1,2,3$. 
 Using lemma \ref{lemageneral}  
we can {construct} a subdivision $T'$ of the triangulation  $T$ compatible with the four curves,  {\it i.e.} such that the union of the four curves is 
a subcomplex of $T'$. By lemma \ref{lemageneral}, the sum $n_0 - n_1 + n_2$ 
remains the same.

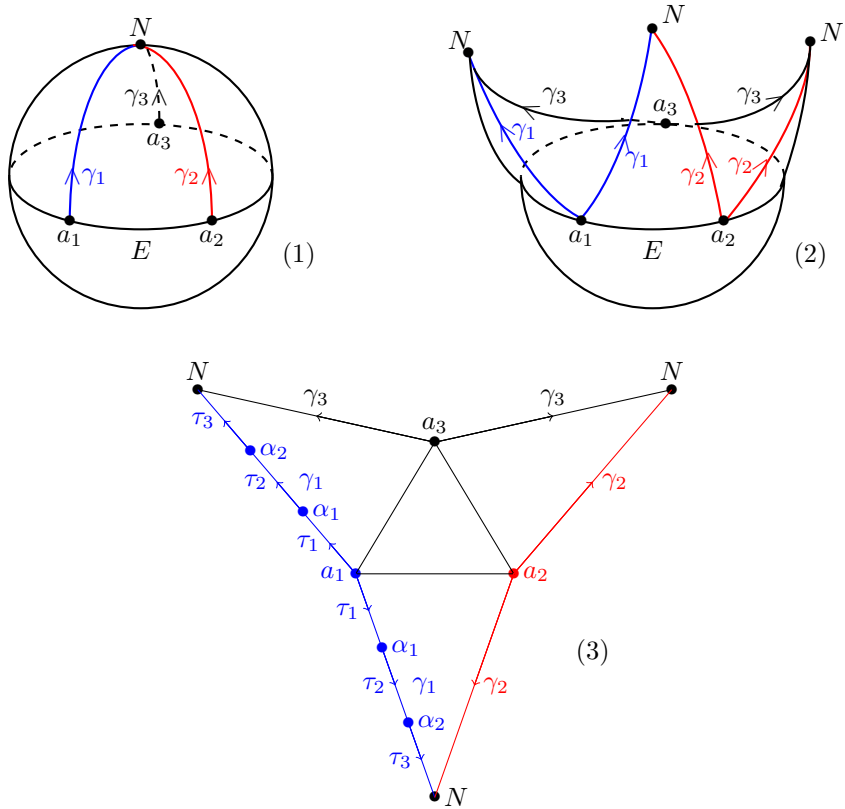
\begin{figure} [H] 
 \begin{tikzpicture}  [scale= 0.35]
 
 \draw[thick] (0,0) circle (5);
 \draw [thick](-5,0) arc (180:360: 5  and 2 );
  \draw [thick,dashed](5,0) arc (0:180: 5  and 2 );

 \draw [blue,thick](0,5) arc (90:180: 2.7  and 6.7 );
 \draw [red,thick](2.7,-1.7) arc (0:90: 3  and 6.7 );
  \draw [thick,dashed](0.7,2.2) arc (0:45: 1.8  and 4 );
  
  \filldraw (0,5) node {$\bullet$} node [above] {$N$};
  \filldraw (-2.7,-1.7) node {$\bullet$} node [below] {$a_1$};
    \filldraw (2.7,-1.7) node {$\bullet$} node [below] {$a_2$};
    \filldraw (0.7,2) node {$\bullet$} node [below] {$a_3$};
\filldraw (0,-2.1)   node [below] {$E$};

\node at (-2.6, 0)[blue]{$\wedge$};
\node at (-2.6,0)[blue,right]{$\gamma_1$};
\node at (2.6, 0)[red]{$\wedge$};
\node at (2.6,0)[red,left]{$\gamma_2$};
\node at (0.7,3){$\wedge$};
\node at (0.7,3)[left]{$\gamma_3$};

\node at (7,-3)[left]{$(1)$};

\end{tikzpicture}
\qquad \qquad
 \begin{tikzpicture}  [scale= 0.35]
 
 \draw[thick] (-5,0) arc (180:360: 5  and 5);
 
 \draw [thick](-5,0) arc (180:360: 5  and 2 );
  \draw [thick,dashed](5,0) arc (0:180: 5  and 2 );

\draw [rotate=-10][blue,thick](-2.5,-2.2)  arc (300:366: 3  and  8 ) node[sloped,midway]{$>$};
 \draw  [rotate=10] [red,thick](2.3,-2) arc (0:57: 3  and 9 ) node[sloped,near start]{$<$};
 \node at (-1.4,0.65)[blue,right]{$\gamma_1$};
  \node at (2.4,0)[red,left]{$\gamma_2$};
  
\draw[rotate=320] [red,thick](3.2,0.3) arc (0:70: 3  and  8 ) ; 
\draw[rotate=340] [thick](4.7,1.3) arc (0:40: 3 and  8 ) ;
\draw [thick](1.6,2) arc (270:360: 4.39 and  3 )node[sloped,midway]{$>$};
\draw [thick,dashed](0.7,2) -- (1.2,2);

\node at (4.2,0.4)[red] [rotate=240]  {$<$};
  \node at (4.2,0.4)[red,left]{$\gamma_2$};
\node at (3.7,2.4)[above]{$\gamma_3$};

  \draw[rotate=20] [blue,thick](-3,-0.65) arc (250:180: 3  and  8 ) ; 
\draw[thick](-5,-0.25) arc (250:187: 3  and  6 );
\draw[rotate=5]  [thick](-6.4,4.7) arc (190: 270: 5.8 and  3 ) node[sloped,midway]{$<$};
\draw [thick,dashed](0.7,2) -- (-1.2,2.2);

\node at (-5.62,1.7)[blue] [rotate=300]  {$<$};
\node at (-3.7,2.4)[above]{$\gamma_3$};
 \node at (-5.7,1.7)[blue,right]{$\gamma_1$};
 
  \filldraw (-2.7,-1.7) node {$\bullet$} node [below] {$a_1$};
    \filldraw (2.7,-1.7) node {$\bullet$} node [below] {$a_2$};
    \filldraw (0.5,2) node {$\bullet$} node [above] {$a_3$};
\filldraw (0,-2.1)   node [below] {$E$};
\filldraw (0,6.2)  node [right] {$N$};
\filldraw (6,5.5)   node [right] {$N$};
\filldraw (-6.5,5.2)   node [left] {$N$};
\filldraw (0,5.6)  node {$\bullet$};
\filldraw (6,5.1)   node {$\bullet$};
\filldraw (-7,4.7)   node {$\bullet$};

\node at (7,-3)[left]{$(2)$};

\end{tikzpicture}
\bigskip
\vglue0.2truecm
 \begin{tikzpicture}  [scale= 0.35]

\coordinate (N1) at  (9,7);
\filldraw (N1) node {$\bullet$} node [above] {$N$};
\coordinate (N2) at  (-9,7);
\filldraw (N2) node {$\bullet$} node [above] {$N$};
\coordinate (N3) at  (0,-8.5);
\filldraw (N3) node {$\bullet$} node [right] {$N$};

\coordinate (a1) at  (-3,0);
\filldraw (a1) node [blue] {$\bullet$} node [blue,left] {$a_1$};
\coordinate (a2) at  (3,0);
\filldraw (a2) node [red]{$\bullet$} node [red,right] {$a_2$};
\coordinate (a3) at  (0,5);
\filldraw (a3) node {$\bullet$} node [above] {$a_3$};

\draw (N2)--(a3) -- (N1);
\draw (a1) --(a3)  -- (a2)  -- (a1);
\draw [blue] (N3) -- (a1) --(N2);
\draw [red](N1) -- (a2) -- (N3);

\coordinate (c1) at  (-5,2.33);
\coordinate (c2) at  (-7,4.66);
\filldraw (c1) node [blue] {$\bullet$} node [blue,right] {$\alpha_1$};
\filldraw (c2) node [blue] {$\bullet$} node [blue,right] {$\alpha_2$};

\coordinate (d1) at (-4,1.165);
\filldraw [blue,->] (a1)-- (d1);
\coordinate (d2) at (-6,3.495);
\filldraw [blue,->] (c1)-- (d2);
\coordinate (d3) at (-8,5.825);
\filldraw [blue, ->] (c2)-- (d3);

\coordinate (e1) at  (-2,-2.833);
\coordinate (e2) at  (-1,-5.666);
\filldraw (e1) node [blue]{$\bullet$} node [blue,right] {$\alpha_1$};
\filldraw (e2) node [blue]{$\bullet$} node [blue,right] {$\alpha_2$};

\coordinate (f1) at (-2.5,-1.416);
\filldraw [blue,->] (a1)-- (f1);
\coordinate (f2) at (-1.5,-4.242);
\filldraw [blue,->] (e1)-- (f2);
\coordinate (f3) at (-0.5,-7.08);
\filldraw [blue,->] (e2)-- (f3);

\filldraw (d1)  node [blue,left] {$\tau_1$};
\filldraw (d2) node [blue,left] {$\tau_2$};
\filldraw (d3)  node [blue,left] {$\tau_3$};
\filldraw (f1)  node [blue,left] {$\tau_1$};
\filldraw (f2) node [blue,left] {$\tau_2$};
\filldraw (f3) node [blue,left] {$\tau_3$};

\coordinate  (g) at (-4.5,6);
\coordinate  (h) at (4.5,6);
\coordinate  (j) at (6,3.5);
\coordinate  (k) at (1.5,-4.25);

\filldraw [->] (a3)-- (g);
\filldraw [->] (a3)-- (h);
\filldraw [red,->] (a2)-- (j);
\filldraw [red,->] (a2)-- (k);

\filldraw (g)  node [above] {$\gamma_3$};
\filldraw (h) node [above] {$\gamma_3$};
\filldraw (j)  node [red,right] {$\gamma_2$};
\filldraw (k)  node [red,right] {$\gamma_2$};
\filldraw (-5.5,3.5) node [blue,right] {$\gamma_1$};
\filldraw (-1.2,-4.25) node [blue,right] {$\gamma_1$};

\node at (7,-3)[left]{$(3)$};
\end{tikzpicture}
\caption{Planar {representation}  of the sphere.}\label{2spheres}
\end{figure}

Now, we cut the sphere along the curves $\gamma_1$, $\gamma_2$ and $\gamma_3$,  
in such a way that the projection provides a polygon $K$ in the plane containing the equator (see figure  \ref{2spheres}). 
Notice that the projection of $T'$ gives us a sub-triangulation $K'$ of $K$. 
We obtain a planar  {representation} of the sphere, homeomorphic to a disc with identifications of the 
simplices on the boundary  $K_0$ of $K$  
corresponding to the cuts.

Theorem \ref{teo1} 
states that the sum $n_0^T - n_1^T + n_2^T$ of the triangulation $T$ 
is equal to the sum  
$n_0^{K_0} - n_1^{K_0} +1 $, 
where the sum  $n_0^{K_0} - n_1^{K_0}$ is calculated by the boundary of the figure. 
Notice that, using the same notation on the sphere and on the planar  {representation}, the vertex $N$ 
is common to all the curves $\gamma_i$ and must be identified. Beside this vertex $N$,  the number of vertices in each curve $\gamma_i$ is equal to the number of edges 
 (see figure \ref{2spheres}).  
Then for the boundary of the planar  {representation}, we have 
$n_0^{K_0} - n_1^{K_0} = +1$  and for the triangulation, we have 
$$n_0^T- n_1^T + n_2^T = +2.$$

\medskip

\subsection{The torus case.}

Let $T$ be a triangulation of the torus $\T = \Sp^1 \times \Sp^1$. 
We choose a meridian $M=\Sp^1 \times \{0\}$ and a parallel $P=\{0\} \times \Sp^1$. 
They cross at one point $A = \{0\} \times\{0\}$. 
Observe that, without loss of generality, 
we can choose them transversally to all the edges 
(1-dimensional simplices of $T$). 
We define a sub-triangulation $T'$ of 
$T$, in the following way (see figure \ref{letore}): 
Each triangle $\sigma$ (2-dimensional simplex) of $T$ 
{meeting} $M$ or $P$  is divided in such the way that $\sigma \cap M$ (or  $\sigma \cap P$)  
is an edge of  $T'$. 
Lemma \ref{lemageneral} 
implies that the sum $n_0 - n_1 + n_2$ remains the same for $T$ and $T'$. 

\begin{figure}[H]
\begin{tikzpicture}[scale=1.2]

\draw[very thick] (3, 1.5) ellipse (3cm and 1.5 cm);
\draw [very thick] (0.8, 1.6) sin (3, 1) cos (5.2, 1.6);
\draw [very thick] (1, 1.5) sin  (3, 2) cos (5, 1.5);); 

\draw[dotted, very thick, red] (3,0) arc (-90:0:0.6) arc (0:90: 0.4cm) ;
\draw[very thick,red] (3,0) arc (90:0:-0.4) arc (0:-90: -0.6cm) ;

\draw[green, very thick] (2, 0.5) -- (4, 0.2) -- (4, 0.8) -- (2, 0.5); 
\node at (4, 0.5) [right] {\textcolor{green}{$K$}};

\draw[very thick,red] (7.5,-2.5) arc (90:0:-0.4) arc (0:-90: -0.6cm) ;
\draw[green, very thick] (6.5, -2) -- (8.5, -2.3) -- (8.5, -1.7) -- (6.5, -2); 
\node at (7.11, -2.14) {$\bullet$}; 
\node at (7.13, -1.9) {$\bullet$};

\draw [green, very thick] (7.11, -2.14) -- (7.13, -1.9); 
\draw [green, very thick] (7.13, -1.9) -- (8.5, -2.3); 
\node at (8.5, -2) [right] {\textcolor{green}{$K'$}};


\draw [->] (6,0) cos ( 7, -1); 
\end{tikzpicture}
\caption{{Sub-triangulation $T'$ of} $T$. }
\label{letore}
\end{figure}
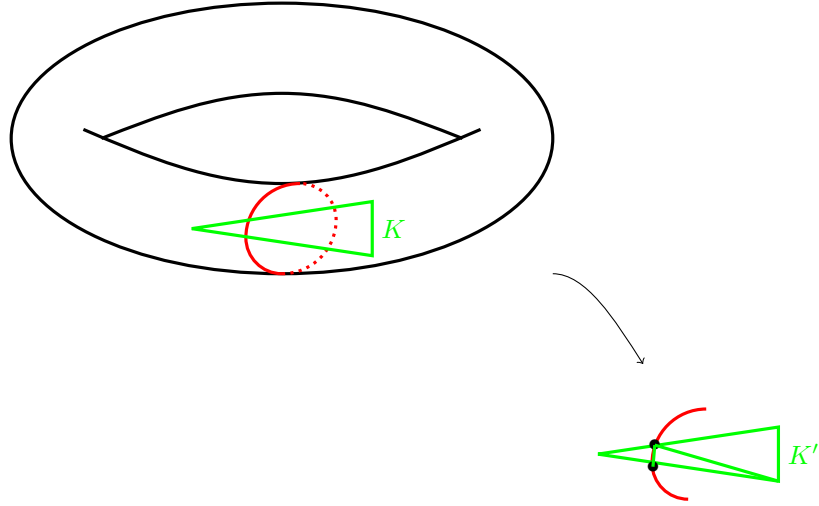

\begin{figure}[H]
\begin{tikzpicture}[scale=0.6]

 \draw [blue, thick]  (0, 0) -- (6,0) ;
 \draw [red, thick]  (6,0) -- (6,6) ; 
 \draw [red, thick] (0,6) -- (0,0); 
\draw [blue, thick]  (6,6) --  (0,6) ;
 
\coordinate  (0) at (0,0);
\coordinate  (10) at (0,1.5);
\coordinate  (11) at (0,2.8);
\coordinate  (12) at (0,3.5);
\coordinate  (13) at (0,6);

\coordinate  (9) at (6,0);
\coordinate  (8) at (6,1.5);
\coordinate  (19) at (6,2.8);
\coordinate  (20) at (6,3.5);
\coordinate  (21) at (6,6);

\coordinate  (1) at (2,0);
\coordinate  (4) at (3,0);
\coordinate  (6) at (4.8,0);

\coordinate  (22) at (2,6);
\coordinate  (23) at (3,6);
\coordinate  (24) at (4.8,6);

\coordinate  (2) at (1.2,2);
\coordinate  (3) at (2.4,2);
\coordinate  (5) at (3.7,2.5);
\coordinate  (7) at (5,2.2);

\coordinate  (14) at (1.5,5);
\coordinate  (15) at (2.5,4);
\coordinate  (16) at (4,3.8);
\coordinate  (17) at (4.6,3);
\coordinate  (18) at (4.9,5);

\draw [green, thick]  (0) -- (2) -- (1) -- (3) -- (2);
\draw [green, thick]  (5) -- (4) -- (3) -- (5) -- (6) -- (7) -- (5);
\draw [green, thick] (9) -- (7) -- (8);
\draw [green, thick] (10) -- (2) -- (11);
\draw [green, thick] (12) -- (14) -- (13);
\draw [green, thick] (14) -- (22) -- (15) -- (14) -- (2) -- (15) -- (3);
\draw [green, thick] (21) -- (18) -- (24) -- (16) -- (18) -- (17) -- (16) -- (23) -- (15) -- (16) -- (5) -- (15);
\draw [green, thick] (18) -- (20) -- (17) -- (19) -- (7) -- (17) -- (5);

\node at (1,0) [blue] {>};
\node at (1,0)[below]{$a$};
\node at (2.5,0) [blue] {>};
\node at (2.5,0)[below]{$b$};
\node at (4,0)[blue] {>};
\node at (4,0)[below]{$c$};
\node at (5.5,0)[blue] {>};
\node at (5.5,0)[below]{$d$};

\node at (1,6)[blue] {>};
\node at (1,6)[above]{$a$};
\node at (2.5,6)[blue] {>};
\node at (2.5,6)[above]{$b$};
\node at (4,6)[blue] {>};
\node at (4,6)[above]{$c$};
\node at (5.5,6)[blue] {>};
\node at (5.5,6)[above]{$d$};

\node at (0, 0.9)[red]{$\wedge$};
\node at (0,0.9)[left]{$e$};
\node at (0, 2.2)[red]{$\wedge$};
\node at (0,2.2)[left]{$f$};
\node at (0, 3.1)[red]{$\wedge$};
\node at (0,3.1)[left]{$g$};
\node at (0, 4.7)[red]{$\wedge$};
\node at (0,4.7)[left]{$h$};

\node at (6, 0.9)[red]{$\wedge$};
\node at (6,0.9)[right]{$e$};
\node at (6, 2.2)[red]{$\wedge$};
\node at (6,2.2)[right]{$f$};
\node at (6, 3.1)[red]{$\wedge$};
\node at (6,3.1)[right]{$g$};
\node at (6, 4.7)[red]{$\wedge$};
\node at (6,4.7)[right]{$h$};

\node at (0, 0)[left]{$A$};
\node at (6, 0)[right]{$A$};
\node at (0, 6)[left]{$A$};
\node at (6, 6)[right]{$A$};

\end{tikzpicture}
\caption{A planar {representation} of torus.}  \label{Toroplanar} 
\end{figure}
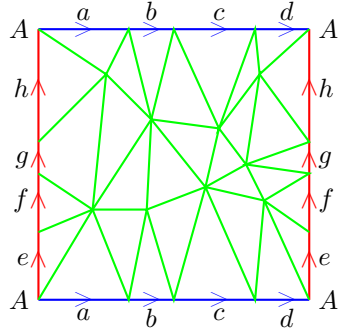

Now, cutting the torus along $M$ and $P$, 
we obtain a planar  {representation} $K$  of the torus which is homeomorphic to a disc, 
 with identifications on the boundary ${K_0}$ corresponding to the cuts. 
 Hence, by lemma \ref{lemageneral}, 
the number $n_0 - n_1 + n_2$ remains unchanged. 
 Using theorem \ref{teo1}, we have:
 $$n_0^T- n_1^T + n_2^T =n_0^{T'} - n_1^{T'} + n_2^{T'} = n_0^K - n_1^K + n_2^K =
 n_0^{{K_0}} - n_1^{{K_0}} + 1.$$ 
Now, with the 
identifications on 
the boundary
 ${K_0}$, we have $n_1^{{K_0}} = n_0^{{K_0}} +1$. 
Finally 
 $$n_0^T- n_1^T + n_2^T = 0$$ 
for any triangulation of the torus.

The same proof holds for the torus of genus $g$. 
For example, let us take a torus of genus 3. 
Given a triangulation  $T$ of the torus, 
we fix a point $A$ and around each ``hole'' of the torus, we fix a ``meridian'' 
($a_1$, $a_2$, $a_3$ in figure \ref{toreg}) and a ``parallel'' ({$b_1$, $b_2$, $b_3$} in {figure \ref{toreg}}). 
We construct a sub-triangulation $T'$ of $T$ by the same 
 methods as in the case of the torus. 
 Cutting the torus of genus $3$ along the meridians and the parallels, 
we obtain a planar  {representation} of the torus of genus $3$, triangulated by the triangulation $T'$. 
By the same 
 procedure as in 
the torus case, one obtains $n_0 - n_1 + n_2 = -4$.  
 This is an example of  
Lhuilier's formula \ref{Lhuilier}.

\begin{figure}[H] 

\begin{tikzpicture}  [scale= 0.17]

 \draw[thick] (14,0) arc (00:155 : 14 and 8.1);
\draw[rotate=120] [thick](9.5,8) arc (0:180: 14 and  8 ) ;
 \draw[rotate=240] [thick](18,10) arc (20:150: 14.5 and  9 ) ;
 
  
  \draw[blue,very thick] plot[smooth] coordinates {(0,-3)(-0.3,0)(-0.9,2.3) (-2.5,3.5) (-4.5,3.8)(-6,3.6)(-7.5,3.1) (-9.5,1.5)
   (-9.5,0)(-7.7,-2)(-4,-2.8) (0,-3)};  
   \draw[blue,very thick] plot[smooth] coordinates {(0,-3)(0.3,0)(0.9,2.3) (2.5,3.5) (4.5,3.8)(6,3.6)(7.5,3.1) (9.5,1.5)
   (9.5,0)(7.7,-2)(4,-2.8) (0,-3)};  
     \draw[blue,very thick] plot[smooth] coordinates {(0,-3)(-3,-7)(-4,-9)(-5.5,-14) (-5,-16)(-2.5,-16.8)(0,-17)
     (2.5,-16.8)(5,-16)(5.5,-14) (4,-9)(3,-7)(0,-3) };   
   
   \draw[very thick,red] plot[smooth] coordinates {(0,-3)(-1.25,-1.6)(-2.5,-0.8)(-4.7,-0.6)};  
    \draw[very thick,dashed,red] plot[smooth] coordinates {(-4.7,-0.6)(-6,-1) (-6.5,-1.2)
    (-9,-2.3)(-13,-4.4)(-14,-5.2)};  
   \draw[very thick,red] plot[smooth] coordinates {(-14,-5.2)(-9.5,-5.6)  (-3,-4.6)  (0,-3)};  
       \draw[very thick,red] plot[smooth] coordinates {(0,-3)(1.25,-1.6)(2.5,-0.8)(4.7,-0.6)};  
    \draw[very thick,dashed,red] plot[smooth] coordinates {(4.7,-0.6)(6,-1) (6.5,-1.2)
    (9,-2.3)(13,-4.4)(14,-5.2)};  
       \draw[very thick,red] plot[smooth] coordinates {(14,-5.2)(9.5,-5.6)  (3,-4.6)  (0,-3)};  
    \draw[very thick,red] plot[smooth] coordinates {(0,-3)(3.5,-6)(7,-11)(7.5,-14)(7.5,-16)(7.2,-17.2)};  
       \draw[very thick,red,dashed] plot[smooth] coordinates {(7.2,-17.2)(6.5,-17.5) (6,-17.8)(5,-17.8)(4,-17.2)
       (3,-16)(2.3,-14)};  
    \draw[very thick,red] plot[smooth] coordinates {(0,-3)(0.2,-7)(1,-11)  (2.3,-14)};  
    
   \coordinate (a1) at (-4.5,3.8);
   \filldraw (a1) node [blue] {$<$} node [above] {$a_1$};
     \coordinate (ra1) at (-4,-2.8);
      
        \coordinate (b1) at (-9.5,-5.6); 
           \filldraw (b1) node {\textcolor{red}{$<$}} node [below] {$b_1$};
     \coordinate (rb1) at (-3,-5.35);

 \coordinate (a2) at (4.5,3.8);
   \filldraw (a2) node [blue] {$>$} node [above] {$a_2$};
     \coordinate (ra2) at (4,-2.8);
      
        \coordinate (b2) at (9.5,-5.6);
   \filldraw (b2) node {\textcolor{red}{$>$}} node [below] {$b_2$};
     \coordinate (rb2) at (3,-5.35);
 
 \coordinate (a3) at (-5.5,-14);
   \filldraw (a3) node [blue] {$\vee$} node [left] {$a_3$};
     \coordinate (ra2) at (4,-2.8);
 
         \coordinate (b3) at (7.5,-13);
   \filldraw (b3) node {\textcolor{red}{$\wedge$}} node [right] {$b_3$};
     \coordinate (rb2) at (3,-5.35);

    \draw [thick](-8,0.6) arc (170:370: 3  and 1 );    
    \draw [thick](-2.45,-0.15) arc (17:163: 2.7  and 1 );
 \draw [thick](2,0.6) arc (170:370: 3  and 1 );
    \draw [thick](7.65,0) arc (17:168: 2.7  and 1 );
        \draw [thick](-3,-13) arc (170:370: 3  and 1 );    
    \draw [thick](2.55,-13.8) arc (17:163: 2.7  and 1 );
  \end{tikzpicture}
  \qquad\qquad
 \begin{tikzpicture}  [scale= 0.55]
 
  \coordinate  (a) at (-1,4) ;
  \coordinate  (b) at (1,4) ;
  \coordinate  (c) at (3,3) ;
  \coordinate  (d) at (4,1) ;
  \coordinate  (e) at (4,-1) ;
 \coordinate  (f) at (3,-3) ;
 \coordinate  (g) at (1,-4) ;
 \coordinate  (h) at (-1,-4) ;
  \coordinate  (j) at (-3,-3) ;
 \coordinate  (k) at (-4,-1) ;
 \coordinate  (l) at (-4,1) ;
 \coordinate  (m) at (-3,3) ;

 \draw [blue]  (a)--(b) node[sloped,midway] {$>$};
 \node at (0,4) [above]{$a_1$};
 
  \draw [red] (b)--(c) node[sloped,midway] {$>$};
  \node at (2,3.8) [right]{$b_1$};
  
 \draw [blue]  (c)--(d) node[sloped,midway] {$<$};
 \node at (4.2,1.7) [above]{$a_1$};  
  
    \draw [red] (d)--(e) node[sloped,midway] {$<$};
   \node at (4,0) [right]{$b_1$}; 
   
 \draw [blue]  (e)--(f) node[sloped,midway] {$<$};
   \node at (4.2,-1.7) [below]{$a_2$};   
   
  \draw[red]  (f)--(g) node[sloped,midway] {$<$};
    \node at (2,-3.8) [right]{$b_2$};
    
 \draw [blue] (g)--(h) node[sloped,midway] {$>$};
      \node at (0,-4) [below]{$a_2$};     
      
        \draw[red]  (h)--(j) node[sloped,midway] {$>$};
     \node at (-2,-3.6) [left]{$b_2$};
  
 \draw [blue]   (j)--(k) node[sloped,midway] {$<$};
  \node at(-4.2,-1.7) [below]{$a_3$};
  
   \draw[red] (k)--(l) node[sloped,midway] {$>$};
 \node at (-4,0) [left]{$b_3$};  
  
 \draw [blue]   (l)--(m) node[sloped,midway] {$<$};
   \node at (-4.2,1.7) [above]{$a_3$}; 
   
     \draw[red] (m)--(a) node[sloped,midway] {$<$};
   \node at (-2,3.8) [left]{$b_3$};   
   
  \end{tikzpicture}
 \caption{Planar {representation} of the torus of genus $3$.
In order to have a light figure, 
we 
did not draw a triangulation $T'$ in such a way that 
each edge  $a_i$ and $b_i$  are subdivided in at least three segments.} \label{toreg}
\end{figure}
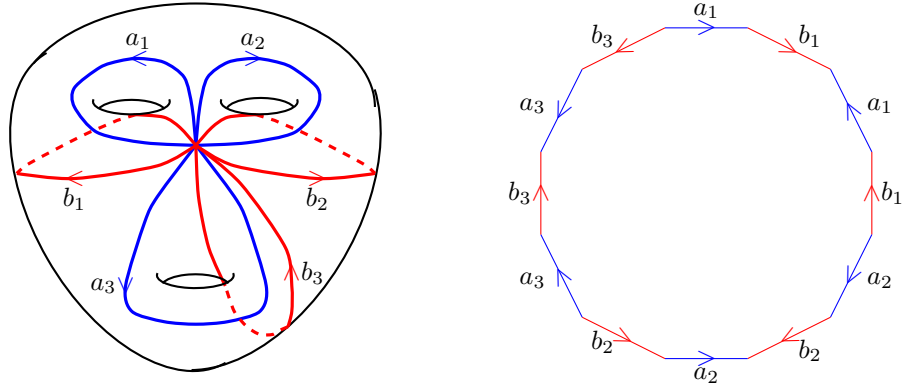

\subsection{The projective plane case.}

The projective plane is represented by a sphere whose diametrically opposite points are identified.
 A triangulation of the projective plane is given by a triangulation of the sphere which is 
{symmetric} with respect to the center of the sphere.
 Let us consider the {sphere} in  $\R^3$ (see 
figure \ref{projo} (a))  and let $T$ be  such a triangulation of the projective plane. 

The intersection of $T$ with the equator defines a triangulation $J$ of the equator 
that is {symmetric} with respect to the center of the sphere. 
Let us define a sub-triangulation $T'$ of $T$ such that simplices  of $J$ 
 are simplices of $T'$ and such that $T'$ is {symmetric} with respect to the center of the sphere 
 (see  figure \ref{projo} (b)). 
By the lemma  \ref{lemageneral}, the sum $n_0 - n_1 + n_2$ is the same for $T$ and $T'$.

\begin{figure}[H]
\begin{tikzpicture}[scale=0.4]

\draw[thick] (0,0) circle (5);
\draw [thick,red](-5,0) arc (180:270: 5  and 2 );
\draw [thick,blue](0,-2) arc (270:360: 5  and 2 );

\draw [thick,dashed,red](5,0) arc (0:90: 5  and 2 );
\draw [thick,dashed,blue](0,2) arc (90:180: 5  and 2 );

\draw [green](1, -4) cos (3.5, -3) sin (3, 0) cos (1, -4);
\draw[dashed, thick,green] (-1, 4) cos (-3.5, 3) sin (-3, 0) cos (-1, 4);
\node at (0,-1.25) {$A$};
\node at (0,2.6) {$A$};
\node at (-5,0) [left] {$B$};
\node at (5,0) [right]{$B$};
\node at (5,-5) {(a)};

\coordinate (A1) at (-5,0);
\coordinate (A2) at (5,0);
\coordinate (A3) at (0,-2);
\coordinate (A4) at (0,2);
\coordinate (A5) at (13,0);
\coordinate (A6) at (23,0);
\coordinate (A7) at (18,-2);
\coordinate (A8) at (18,2);

 \foreach \point in {A1,A2,A3,A4,A5,A6,A7,A8}
\fill [black,opacity=.5] (\point) circle (5pt);

\draw[thick] (18,0) circle (5);

\draw [thick,red](13,0) arc (180:270: 5  and 2 );
\draw [thick,blue](18,-2) arc (270:360: 5  and 2 );

\draw [thick,dashed,red](23,0) arc (0:90: 5  and 2 );
\draw [thick,dashed,blue](18,2) arc (90:180: 5  and 2 );

\draw [green](19, -4) cos (21.5, -3) sin (21, 0) cos (19, -4);
\draw[dashed, thick,green] (17, 4) cos (14.5, 3) sin (15, 0) cos (17, 4);

\draw[thick, green]  (19.7,-1.95) sin (21.5, -3);
\draw[thick, green, dashed]  (16.3,1.95) sin (14.5, 3);
\node at (18,-1.25) {$A$};
\node at (18,2.6) {$A$};
\node at (13,0) [left] {$B$};
\node at (23,0) [right]{$B$};
\node at (23,-5) {(b)};

\end{tikzpicture} 
\caption{Triangulation $T$ of the projective plane \quad  Sub-triangulation $T'$ of $T$.}\label{projo}
\end{figure}
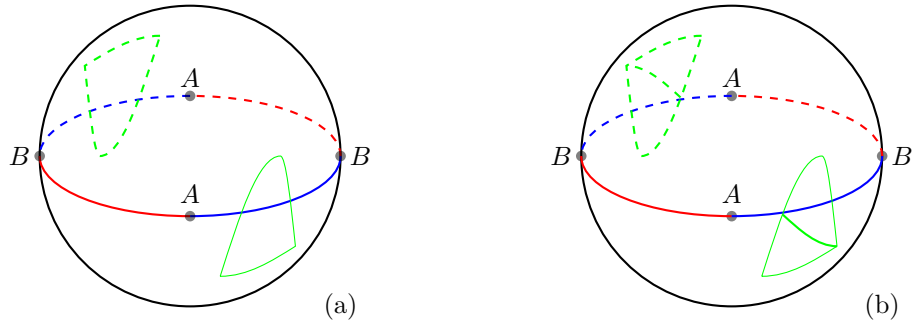

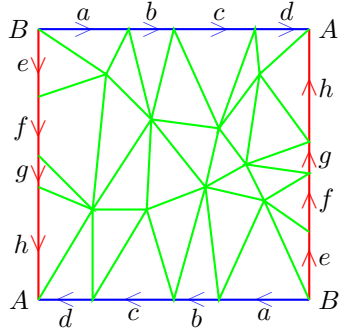
\begin{figure}[H]
\begin{tikzpicture}[scale=0.6]

 \draw [blue, thick]  (0, 0) -- (6,0) ;
 \draw [red, thick]  (6,0) -- (6,6) ; 
 \draw [red, thick] (0,6) -- (0,0); 
\draw [blue, thick]  (6,6) --  (0,6) ;
 
\coordinate  (0) at (0,0);
\coordinate  (25) at (0,4.5);
\coordinate  (26) at (0,3.2);
\coordinate  (27) at (0,2.5);
\coordinate  (13) at (0,6);

\coordinate  (9) at (6,0);
\coordinate  (8) at (6,1.5);
\coordinate  (19) at (6,2.8);
\coordinate  (20) at (6,3.5);
\coordinate  (21) at (6,6);

\coordinate  (28) at (4,0);
\coordinate  (4) at (3,0);
\coordinate  (29) at (1.2,0);

\coordinate  (22) at (2,6);
\coordinate  (23) at (3,6);
\coordinate  (24) at (4.8,6);

\coordinate  (2) at (1.2,2);
\coordinate  (3) at (2.4,2);
\coordinate  (5) at (3.7,2.5);
\coordinate  (7) at (5,2.2);

\coordinate  (14) at (1.5,5);
\coordinate  (15) at (2.5,4);
\coordinate  (16) at (4,3.8);
\coordinate  (17) at (4.6,3);
\coordinate  (18) at (4.9,5);

\draw [green, thick]  (0) -- (2) -- (29) -- (3) -- (2);
\draw [green, thick]  (5) -- (4) -- (3) -- (5) -- (28) -- (7) -- (5);
\draw [green, thick] (9) -- (7) -- (8);
\draw [green, thick] (27) -- (2) -- (26);
\draw [green, thick] (25) -- (14) -- (13);
\draw [green, thick] (14) -- (22) -- (15) -- (14) -- (2) -- (15) -- (3);
\draw [green, thick] (21) -- (18) -- (24) -- (16) -- (18) -- (17) -- (16) -- (23) -- (15) -- (16) -- (5) -- (15);
\draw [green, thick] (18) -- (20) -- (17) -- (19) -- (7) -- (17) -- (5);

\node at (5,0) [blue] {<};
\node at (5,0)[below]{$a$};
\node at (3.5,0) [blue] {<};
\node at (3.5,0)[below]{$b$};
\node at (2.1,0)[blue] {<};
\node at (2.1,0)[below]{$c$};
\node at (0.6,0)[blue] {<};
\node at (0.6,0)[below]{$d$};

\node at (1,6)[blue] {>};
\node at (1,6)[above]{$a$};
\node at (2.5,6)[blue] {>};
\node at (2.5,6)[above]{$b$};
\node at (4,6)[blue] {>};
\node at (4,6)[above]{$c$};
\node at (5.5,6)[blue] {>};
\node at (5.5,6)[above]{$d$};

\node at (0, 5.2)[red]{$\vee$};
\node at (0,5.2)[left]{$e$};
\node at (0, 3.8)[red]{$\vee$};
\node at (0,3.8)[left]{$f$};
\node at (0, 2.8)[red]{$\vee$};
\node at (0,2.8)[left]{$g$};
\node at (0, 1.25)[red]{$\vee$};
\node at (0,1.25)[left]{$h$};

\node at (6, 0.9)[red]{$\wedge$};
\node at (6,0.9)[right]{$e$};
\node at (6, 2.2)[red]{$\wedge$};
\node at (6,2.2)[right]{$f$};
\node at (6, 3.1)[red]{$\wedge$};
\node at (6,3.1)[right]{$g$};
\node at (6, 4.7)[red]{$\wedge$};
\node at (6,4.7)[right]{$h$};

\node at (0, 0)[left]{$A$};
\node at (6, 0)[right]{$B$};
\node at (0, 6)[left]{$B$};
\node at (6, 6)[right]{$A$};

\end{tikzpicture}
\caption {A planar {representation} of the projective plane.}  \label{Projplanar2} 
\end{figure}

Now, the {orthogonal} projection of the northern 
hemisphere 
 to the plane $0xy$ 
provides a triangulation of the disc $D$ of radius $1$, {centered} at the origin,  
whose triangulation of the boundary is {symmetric} with respect to the center of the disc.  
With the identification of simplices, we have 
 $n_0 - n_1 = 0$ on the boundary. 
Then, by theorem \ref{teo1}, we have 
 $$n_0^T- n_1^T + n_2^T = +1$$
 for any triangulation of the projective plane.

\subsection{The Klein bottle case.}

The case of Klein bottle is similar to the case of torus. 
Given a triangulation $T$ of the Klein bottle, 
we choose a meridian $M$ and a parallel $P$. 
Let us define a sub-triangulation $T'$ of $T$ compatible with $M$ and $P$. 
The cut along $M$ and $P$ provides a planar {representation} of the Klein bottle 
as a rectangle triangulated with identifications on the boundary. 
On the boundary, we have 
$n_1  = n_0  +1$. 
Then, by theorem \ref{teo1}, we have 
 $$n_0^T- n_1^T + n_2^T = 0$$
for any triangulation of the Klein bottle.

\begin{figure}[H]
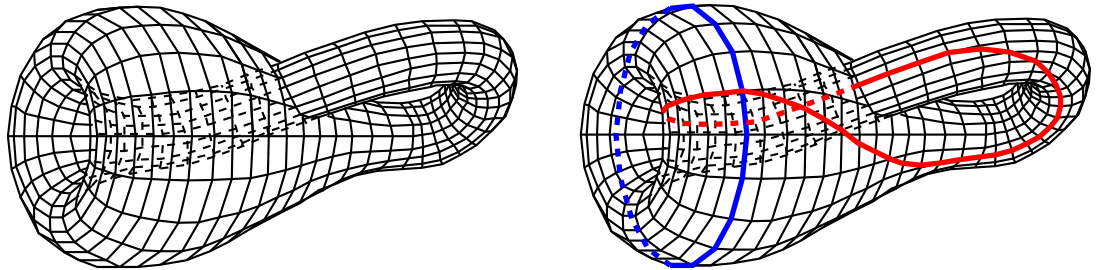
\centering 

  \caption{The Klein bottle and cuts (the meridian $M$ is in blue and the parallel $P$ is in red).}\label{Klein1}
\end{figure}

\begin{figure}[H]
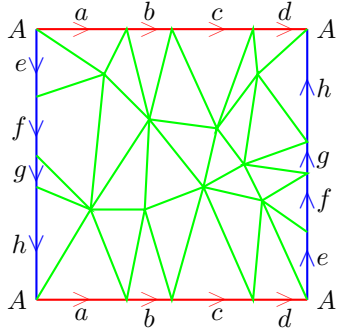

%

\caption{ A planar {representation} of the Klein bottle.} 
\end{figure}

\subsection{The pinched torus case.}

Not every surface with singularities admits a planar {representation}. 
The pinched torus is an example of a singular surface 
that does admit such a planar {representation}.

Let us recall that the pinched torus is a surface in $\R^3$ defined by the following  cartesian 
{parameterization}:
\begin{equation*}
\begin{cases}
x = \left( r_1 + r_2 \; {\rm cos} (v) \, {\rm cos} \left(\frac{1}{2} u\right) \right) \cos (u) \\
y = \left( r_1 + r_2 \; \cos(v) \, \cos\left(\frac{1}{2} u\right) \right) \sin (u) \\
z= r_2\; \sin (v)\,  \cos \left(\frac{1}{2} u\right)  
 \end{cases}
\end{equation*}
where $r_1$ and $r_2$ are respectively the 
 large and small radii.

\begin{figure}[H]
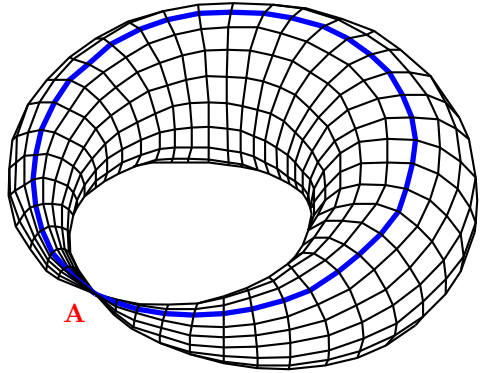
 \centering 

\caption{The pinched torus and cut.}\label{toropincado}
\end{figure}

Let $T$ be a triangulation of the pinched torus.  
We choose a ``parallel'' $P$ passing 
through the singular point $A$
 of the pinched torus and 
 we define a sub-triangulation $T'$ of $T$ compatible with $P$ using lemma \ref{lemageneral}.  
By cutting along  $P$, one obtains a planar {representation} of the pinched torus (figure \ref{pincado})  with {identifications} on the boundary. 
One observes that the point $A$ is duplicated. 
On  the boundary, we have 
$n_0 - n_1 = 0$. 
Then, by theorem \ref{teo1}, we have
 $$n_0^T- n_1^T + n_2^T = +1$$ 
for any triangulation of the pinched torus.

\begin{figure}[H]
\begin{tikzpicture}[scale=0.8]

 \draw [blue, thick]  (1,0) -- (-2,3) -- (1,6) ;
 \draw [blue, thick]  (1,6) -- (4,6) ;
\draw [blue, thick]  (1,0) --  (4,0) ;
 \draw [blue, thick]  (4,0) -- (7,3) -- (4,6) ;
 
 \node at (-1.5,3.5) [blue,rotate=40]  {$>$};
  \node at (-0.5,4.5) [blue,rotate=40]  {$>$};
   \node at (0.5,5.5) [blue,rotate=40]  {$>$};
    \node at (1.5,6) [blue]  {$>$};
     \node at (2.7,6) [blue]  {$>$};
      \node at (3.75,6) [blue]  {$>$};
       \node at (4.8,5.2) [blue,rotate=-40]  {$>$};
        \node at (5.8,4.2) [blue,rotate=-40]  {$>$};
         \node at (6.5,3.5) [blue,rotate=-40]  {$>$};
 \node at (-1.5,2.5) [blue,rotate=-40]  {$>$};
  \node at (-0.5,1.5) [blue,rotate=-40]  {$>$};
   \node at (0.5,0.5) [blue,rotate=-40]  {$>$};
    \node at (1.5,0) [blue]  {$>$};
     \node at (2.7,0) [blue]  {$>$};
      \node at (3.75,0) [blue]  {$>$};
       \node at (4.8,0.8) [blue,rotate=40]  {$>$};
        \node at (5.8,1.8) [blue,rotate=40]  {$>$};
         \node at (6.5,2.5) [blue,rotate=40]  {$>$};
                  
 \node at (-2.1,3.1) [red,above]  {$\bf A$};
 \node at (7.1,3.1) [red,above]  {$\bf A$};
  \node at (-2,3) [red]  {$\bullet$};
 \node at (7,3) [red]  {$\bullet$};

 \coordinate  [label=above:$a$](9) at (-1,4);
 \coordinate [label=above:$b$] (13) at (0,5);
\coordinate  [label=above:$c$](26) at (1,6); 
\coordinate  [label=above:$d$](22) at (2,6);
\coordinate  [label=above:$e$](23) at (3.5,6);
\coordinate [label=above:$f$] (24) at (4,6);
\coordinate [label=above:$g$] (18) at (5.5,4.5);
\coordinate  [label=above:$h$](20) at (6,4);

\coordinate  [label=below:$a$](27) at (-1,2);
\coordinate  [label=below:$b$] (0) at (0,1);
\coordinate [label=below:$c$] (30) at (1,0);
\coordinate  [label=below:$d$](1) at (2,0);
\coordinate [label=below:$e$] (4) at (3.5,0);
\coordinate  [label=below:$f$](6) at (4,0);
\coordinate [label=below:$g$] (8) at (5.5,1.5);
\coordinate  [label=below:$h$](19) at (6,2);

\coordinate  (25) at (0,3);
\coordinate   (21) at (5.5,3);

\coordinate  (2) at (1.2,2);
\coordinate  (3) at (2.4,2);
\coordinate  (5) at (3.7,2.5);
\coordinate  (7) at (5,2.2);

\coordinate (14) at (1.5,5);
\coordinate (15) at (2.5,4);
\coordinate (16) at (4,3.8);
\coordinate  (17) at (4.6,3);

\draw [green, thick]  (-2,3) -- (25) ;
\draw [green, thick]  (27) -- (25) -- (9);
\draw [green, thick]  (27) -- (25) -- (9);
\draw [green, thick]  (13) -- (25);
\draw [green, thick]  (21) -- (7,3);
\draw [green, thick] (20) -- (21) -- (17);
\draw [green, thick] (19) -- (21) -- (7);
\draw [green, thick] (6) -- (7) -- (5);
\draw [green, thick]  (2) -- (30);
\draw [green, thick]  (14) -- (26);
\draw [green, thick]  (0) -- (2) -- (1) -- (3) -- (2);
\draw [green, thick]  (5) -- (4) -- (3) -- (5) -- (6); 
\draw [green, thick] (7) -- (8);
\draw [green, thick] (27) -- (2) -- (25);
\draw [green, thick] (25) -- (14) -- (13);
\draw [green, thick] (14) -- (22) -- (15) -- (14) -- (2) -- (15) -- (3);
\draw [green, thick] (24) -- (16) -- (18) -- (17) -- (16) -- (23) -- (15) -- (16) -- (5) -- (15);
\draw [green, thick] (20) -- (17);
\draw [green, thick] (19) -- (7) -- (17) -- (5);

\end{tikzpicture}
\caption{A planar {representation} of the pinched torus.} \label{pincado}
\end{figure}
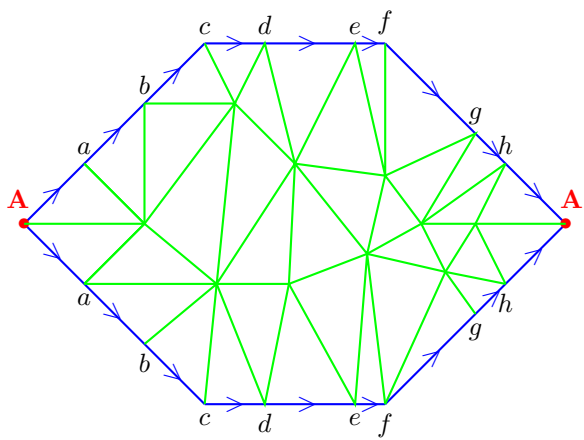

\def\refname{References}


\begin{thebibliography}{20}



\bibitem [Bra]{Bra} H.R. Brahana, 
{\it Systems of circuits on two-dimensional manifolds}, 
Ann. Math. 23 (1921), 144-168. 

\bibitem[BT]{BT}  J.-P. Brasselet e Thuy N.T.B., {\it Teorema de Poincar\'e-Hopf}, C.Q.D.-
Revista {Electr\^onica} Paulista de Matem\'atica. Bauru, v. 16, pp. 134--162, dez. 2019.

\bibitem [BW] {BW} G.E. Bredon and  J. W.  Wood, {\it Non-orientable surfaces in  orientable 3-manifolds,} 
Inventiones Mathematicae, Vol. 7, no 2, 1969, pp. 83---110. 

\bibitem [BM]{BM} H. Bruggesser and P. Mani {\it Shellable decompositions of cells and spheres,} 
Math. Scand. 29 (1971), 197-205.

\bibitem [Ca1] {Ca1} A.L. Cauchy, {\it Recherches sur les poly\`edres,} {$\rm I^{er}$} M\'emoire lu \`a la premi\`ere classe de l'Institut, en F\'evrier 1811, par A.L. Cauchy, Ing\'enieur des Ponts et Chauss\'ees. Journal de l'Ecole Polytechnique, Volume 9, (1913) 68 -- 86. 


\bibitem [Ca2] {Ca2} A.L. Cauchy, {\it Sur les polygones et  les poly\`edres,} Second M\'emoire  lu \`a la premi\`ere classe de l'Institut, le 20 Janvier 1812, par A.L. Cauchy, Ing\'enieur des Ponts et Chauss\'ees. Journal de l'Ecole Polytechnique, Volume 9, (1913) 87 -- 98.

\bibitem[Cel]{Cel} Cellule de G\'eom\'etrie, {\it Relation d'Euler et les poly\`edres sans ``trou'',}  Equipe de la Cellule de G\'eom\'etrie de l'UMONS et de la Haute Ecole de la Communaut\'e Fran\c caise en Hainaut. \hfill\break
\url{http://www.cellulegeometrie.eu/documents/pub/pub_12.pdf}


\bibitem [CR]{CR} R. Courant and H. Robbins, {\it What Is Mathematics?: An Elementary Approach to Ideas and Methods}, London: Oxford University Press. 1941. 
2nd edition, with additional material by Ian Stewart. New York: Oxford University Press. 1996. 

\bibitem [dJ1] {dJ1} E. de Jonqui\`eres, {\it Note sur un M\'emoire de Descartes longtemps in\'edit et sur les titres de son auteur \`a la priorit\'e d'une d\'ecouverte dans la th\'eorie des poly\`edres.} Compte-Rendus des S\'eances de l'Acad\'emie des Sciences 110, 1890, pp. 261--266.

\bibitem [dJ2] {dJ2} E. de Jonqui\`eres, {\it \'Ecrit 
posthume de Descartes sur les poly\`edres},  Compte-Rendus des S\'eances de l'Acad\'emie des Sciences 110, 1890, pp. 315--317.

\bibitem [dJ3] {dJ3} E. de Jonqui\`eres, {\it 
\'Ecrit posthume de Descartes. De solidorum elementis.  Texte latin (original et revu) suivi d'une traduction fran\c caise avec notes}, M\'emoire pr\'esent\'e \`a l'Acad\'emie des sciences dans sa s\'eance du 31 mars 1890, s\'erie 45, pp. 325--379.

\bibitem [De] {De} R. Descartes, {\it De solidorum elementis}, \OE uvres de Descartes, 
vol 10, Paris 1908, pp. 256--276.

\bibitem[DH]{DH} M. Dehn und P. Heegard, {\it Analysis Situs}, Enzyklop\"adie der Mathematischen Wissenschaften, III AB 3, 153--220 (Leipzig, Teubner, 1907). 

\bibitem[vD]{vD} W. von Dyck, {\it Beitr\"age zur Analysis Situs I}, Mathematische Annalen, 32, pp. 457--512 (1888). 

\bibitem [Epp]{Epp} D. Eppstein, {\it Geometry Junkyard Twenty Proofs of Euler Formula}, \hfill\break
\url{https://www.ics.uci.edu/~eppstein/junkyard/euler/}

\bibitem [Eu1]{Eu1} L. Euler, {\it Leonard Euler und Christian Goldbach, Briefwechsel, 1729-1764,} Heransgegeben und eingeleitet von A.P. ${\rm Ju{\check s}kevi{\check c}}$ und E. Winter, Berlin Akademie-Verlag 1965.

\bibitem [Eu2]{Eu2} L. Euler, {\it Elementa doctrinae solidorum,} 1750. Novi Comment. Acad. Sc. Petrop. t.4 
Saint P\'etersbourg, 1758, pp. 72--93. 

\bibitem [Eu3]{Eu3} L. Euler, {\it  Demonstratio nonnularum 
insignium proprietatum, quibus solida hedris planis inclusa sunt praedita,} 1751. 
Novi Comment. Acad. Sc. Petrop. t.4  Saint P\'etersbourg, 1758, pp. 94--108. 

\bibitem [Eve]{Eve} H. Eves, {\it An introduction to the History of Mathematics,} 1953. 
Saunders Series (6th edition 1990).

\bibitem [Fo] {Fo} L.-A. Foucher de Careil, {\it \OE uvres in\'edites de Descartes pr\'ec\'ed\'ees d'une introduction
sur la m\'ethode}, 2 volumes, Paris, 1859, 1860.

\bibitem [FR]{FR} C. Francese and D. Richardson, {\it The flaw in Euler's proof of his polyhedral formula}, 
The American monthly, vol 114, no 4, april 2007, pp. 286--296. 

\bibitem [GX] {GX} J. H. Gallier and D. Xu, {\it A Guide to the Classification Theorem for Compact Surfaces}, 
 Geometry and Computing Book 9, Springer, 2013.
 
 \bibitem [Ha] {Ha} A. Hatcher, {\it Algebraic Topology}, Cambridge University Press, 2002.
 
 \bibitem [HC] {HC} D. Hilbert und S. Cohn-Vossen, {\it  Anschauliche Geometrie,} Springer, (1932), English version {\it Geometry and the Imagination} AMS Chelsea Publishing, (1952).
 
 \bibitem[Jo]{Jo} C. Jordan, {\it Sur la d\'eformations des surfaces}, Journal de Math\'ematiques Pures et Appliqu\'ees, 2\`eme s\'erie, tome XI, pp. 105--109, 1866.
 
   \bibitem[Kirk]{Kirk} A. Kirk, {\it Euler's polyhedron formula,} Plus Magazine... 
   living mathematics. 1997--2009, Millennium Mathematics Project, University of Cambridge.\hfill\break
\url{https://plus.maths.org/content/os/issue43/features/kirk/index}

 \bibitem[Lak]{Lak} I. Lakatos, {\it Proofs and Refutations,} Cambridge University Press (1976).
 
 \bibitem[Leb]{Leb} H. Lebesgue, {\it Remarques sur les deux premi\`eres d\'emonstrations du 
 th\'eor\`eme d'Euler relatif aux poly\`edres}, Bulletin de la S. M. F., Tome  52 (1924), pp 315-336.
 
 \bibitem[Leg]{Leg}  A. M. Legendre, {\it \'El\'ements de g\'eom\'etrie}, Paris, 1794.
 
\bibitem[Lhu]{Lhu} S.A.J. Lhuilier, {\it M\'emoire sur la poly\'edrom\'etrie, contenant une 
d\'emonstration du th\'eor\`eme d'Euler sur les poly\`edres, et un examen de diverses exceptions auxquelles ce th\'eor\`eme est assujetti,} Gergonne Ann. Math. 3, 1812, p.169. 

\bibitem [Lie] {Lie} W. Lietzmann, {\it Visual Topology}, (translated from the German by 
M. Bruckheimer), American Elsevier, NY 1969, Chatto and Windus, London 1965.

\bibitem[Li]{Li} E. Lima, {\it A caracter\'istica de Euler-Poincar\'e,} 
Matem\'atica Universit\'aria,  Sociedade Mathematica Brasileira, No 1, Junho 1985, pp. 47--62.

\bibitem[Li2]{Li2} E. Lima, {\it O Teorema de Euler sobre Poliedros,} 
Matem\'atica Universit\'aria,  Sociedade Mathematica Brasileira, No 2, Dezembro 1985, pp. 57--74.

\bibitem[Mal1] {Mal1} J. Malkevitch, {\it  Euler's Polyhedral Formula}, Dec. 2004, AMS Feature Column. Monthly essays on mathematical topics. 

\bibitem[Mal2] {Mal2} J. Malkevitch, {\it  Euler's Polyhedral Formula II}, Jan. 2005, AMS Feature Column. Monthly essays on mathematical topics. 

\bibitem[Mal3] {Mal3} J. Malkevitch, {\it Gifts from Euler's Polyhedral Formula}, Graph Theory Notes, York College. \hfill\break
\url{ https://www.york.cuny.edu/~malk/Mymaterials/gifts2.pdf }


\bibitem[Mas]{Mas} W. S. Massey, {\it A Basic Course in Algebraic Topology,}
 New York: Springer-Verlag, 1997. 

\bibitem[Mo]{Mo} A.F. M\"obius, {\it Theorie der elementaren Verwandschaft}, Abh. der K\"on. 
S\"achsische Ges. der Wiss. 15, pp. 18--57 (1863).

\bibitem [Po2] {Po3} H. Poincar\'e, {\it Sur l'Analysis situs,} 
CRAS  vol. 115 (1892), pp 633-636.

\bibitem [Po3] {Po1} H. Poincar\'e, {\it Sur la g\'en\'eralisation d'un th\'eor\`eme
d'Euler relatif aux poly\`edres}, CRAS  vol. 117 (1893), pp. 144-145.

\bibitem [Rad]{Rad} T. Rad\'o, {\it \"Uber den Begriff der Riemannschen Fl\"ache}, Acta Litt. Sci. Szeged. 2 (1925), 101 -- 121. 

\bibitem [Ri]{Ri} D. Richeson, {\it Euler's Gem: The Polyhedron Formula and the Birth of Topology},  
Princeton University Press, 2008.

\bibitem [Sam]{Sam} H. Samelson, {\it In defense of Euler}, Enseign. Math. (2) 42 (3-4), 1996, 377-382

\bibitem[ST]{ST} H. Seifert und W. Threlfall, {\it Lehrbuch der Topologie}. VII,  Leipzig u. Berlin 1934, B. G. Teubner. English version  {\it A Textbook of Topology}, Pure and Applied Mathematics,  
A Series of Monographs and Textbooks, Academic Press, 1980. 

\bibitem[Sin]{Sin} D.A. Singer, {\it Geometry: Plane and Fancy,} Undergraduate Texts in Mathematics, 1998, Springer. 

\bibitem [Zie]{Zie}  G. M. Ziegler, {\it Lectures on polytopes,} Graduate Texts in Mathematics, Vol. 152. 
Springer-Verlag, New York, 1995.

\end{thebibliography}
\end{document}